\documentclass[9pt,reqno,a4]{amsart}

\numberwithin{equation}{section}

\usepackage[latin1]{inputenc}
\usepackage[english]{babel}

\usepackage{amsmath,amsthm,amsfonts,latexsym,amssymb}
\usepackage[colorlinks]{hyperref}
\hypersetup{linkcolor=blue,citecolor=blue,filecolor=black,urlcolor=blue}
\usepackage{comment}
\usepackage{soul}

\usepackage[
  hmarginratio={1:1},     
  vmarginratio={1:1},     
  textwidth=15cm,        
  textheight=22cm,  
  heightrounded,          
]{geometry}

{ \theoremstyle{plain}
\newtheorem{theorem}{Theorem}[section]
\newtheorem{proposition}[theorem]{Proposition}
\newtheorem{lemma}[theorem]{Lemma}
\newtheorem{corollary}[theorem]{Corollary}
\newtheorem{claim}{Claim}
  \theoremstyle{remark}

  \theoremstyle{definition}

}

\newcommand\R{\text{I\!R}}
\newcommand\N{\text{I\!N}}

\newcommand\e{\epsilon}
\newcommand\ep{\epsilon}
\newcommand\de{\delta}
\newcommand\be{\beta}
\newcommand\al{\alpha}

\newcommand{\Om}{\Omega}
\newcommand{\oen}{\Omega_{\e_n}}
\newcommand{\fr}{\partial}

\newcommand{\grad}{\nabla}
\newcommand{\ml}{\mathcal}

\newcommand{\sm}{\setminus}
\newcommand{\la}{\lambda}
\newcommand{\dem}{\bf Proof:}

\newcommand\lap{\Delta}

\newcommand\ti{\tilde}

\newcommand{\lf}{\left}
\newcommand{\rg}{\right}

\newcommand\ds{\displaystyle}

\renewcommand\({\left(}
\renewcommand\){\right)}

\DeclareMathAlphabet{\mathpzc}{OT1}{pcz}{m}{it}

\begin{document}   \title[On the mean field equation   with  variable intensities on pierced domains]{On the mean field equation   with  variable intensities on pierced domains}

\author[P. Esposito]{Pierpaolo Esposito}
\address{Pierpaolo Esposito 
\newline \indent Universit\`a degli Studi Roma Tre 
\newline \indent Dipartimento di Matematica e Fisica 
\newline \indent L.go S. Leonardo Murialdo  1 
\newline \indent 00146 Roma, Italy}
\email{esposito@mat.uniroma3.it}

\author[P. Figueroa]{Pablo Figueroa}
\address{Pablo Figueroa 
\newline \indent Universidad Cat\'olica Silva Henr\'iquez
\newline \indent Facultad de Educaci\'on
\newline \indent Escuela de Investigaci\'on y Postgrado
\newline \indent General Jofr\'e 462, Santiago, Chile }
\email{pfigueroas@ucsh.cl }

\author[A. Pistoia]{Angela Pistoia}
\address{Angela Pistoia \newline \indent Universit\`a di Roma ``La Sapienza'' \newline \indent
Dipartimento SBAI,  via Antonio Scarpa 16, \newline 
\indent 00161 Roma, Italy}
\email{angela.pistoia@uniroma1.it}

\date{\today}


\maketitle


\begin{abstract}
\noindent We consider the two-dimensional mean field equation of the equilibrium turbulence with variable intensities and Dirichlet boundary condition on a pierced domain 
$$\left\{ \begin{array}{ll} 
-\lap u=\la_1\dfrac{V_1 e^{u}}{ \int_{\Om_{\boldsymbol\epsilon}} V_1  e^{u} dx } - \la_2\tau \dfrac{ V_2 e^{-\tau u}}{ \int_{\Om_{\boldsymbol\epsilon}}V_2 e^{ - \tau u} dx}&\text{in $\Om_{\boldsymbol\epsilon}=\Om\sm \displaystyle \bigcup_{i=1}^m \overline{B(\xi_i,\e_i)}$}\\
\ \ u=0 &\text{on $\fr \Om_{\boldsymbol\epsilon}$},
\end{array} \right. $$
where   $B(\xi_i,\e_i)$ is a ball centered at $\xi_i\in\Omega$ with radius $\e_i$, $\tau$ is a positive parameter and $V_1,V_2>0$ are smooth potentials. When $\lambda_1>8\pi m_1$ and $\lambda_2 \tau^2>8\pi (m-m_1)$ with $m_1 \in \{0,1,\dots,m\}$, there exist radii $\epsilon_1,\dots,\epsilon_m$ small enough such that the problem has a solution  which blows-up positively and negatively  at the points $\xi_1,\dots,\xi_{m_1}$ and $\xi_{m_1+1},\dots,\xi_{m}$, respectively, as the radii approach zero.
\end{abstract}

\tableofcontents
\date{\today} 
\section{Introduction}

\noindent In the pioneering paper \cite{o} Onsager introduced an approach to explain the formation of stable large-scale vortices, which in the context of the statistical mechanics description of 2D-turbulence allowed Caglioti, Lions, Marchioro, Pulvirenti \cite{clmp} and Sawada, Suzuki \cite{ss} to derive the following equation:
\begin{equation}\label{p1}
\left\{\begin{array}{ll}
\ds -\Delta u=\lambda \int\limits_{[-1,1]}{\alpha e^{\alpha u}\over \int\limits_\Omega e^{\alpha u}dx}d \mathcal P(\alpha)& \hbox{in}\ \Omega \\
 u=0 & \hbox{on}\ \partial\Omega,
\end{array}\right.
\end{equation}
where $\Omega$ is a bounded domain in $\mathbb R^2,$ $u$ is the stream function of the flow, $\lambda>
0$ is a constant related to the inverse temperature and  $\mathcal P$ is a Borel probability measure in $[-1,1]$ describing the point-vortex intensities distribution.

\medskip \noindent When $\mathcal P=\delta_1$ is concentrated at $1$, then \eqref{p1} reduces to the classical mean field equation
\begin{equation}\label{p2}
\left\{\begin{array}{ll} 
\ds -\Delta u=\lambda  { e^{ u}\over \int\limits_\Omega e^{  u}dx}& \hbox{in}\ \Omega\\
 u=0& \hbox{on}\ \partial\Omega,
 \end{array}\right.
 \end{equation}
which has been widely studied in the last decades (see the survey \cite{l}). In particular, solutions are critical points of the functional
$$J_\la(u)={1\over 2}\int_\Om |\nabla u|^2-\la \log\lf(\int_\Om e^{u}\rg),\quad u\in H_0^1(\Om).$$
By Moser-Trudinger's inequality solutions can be found as minimizers of $J_\la$ if $\la<8\pi$. In the supercritical regime $\la\ge 8\pi$, the situation becomes subtler since the existence of solutions could depend on the topology and the geometry of the domain. Using a degree argument Chen and Lin \cite{CL1,CL2} proved that \eqref{p2} has a solution when $\la\notin 8\pi\N$ and $\Om$ is not simply connected. On Riemann surfaces the degree argument in \cite{CL1,CL2} is still available and has received a variational counterpart in \cite{Dja,Mal} by means of  improved forms of the Moser-Trudinger inequality. When $\la=8\pi$ problem \eqref{p2} is solvable on a long and thin rectangle, as showed by Caglioti et al. \cite{clmp1}, but not on a ball. Bartolucci and Lin \cite{BL} proved that \eqref{p2} has a solution for $\la = 8\pi $ when the Robin function of $\Om$ has more than one maximum point.

\medskip \noindent When $\mathcal P=\sigma \delta_{1}+(1-\sigma)\delta_{-\tau }$ with $\tau \in[-1,1]$ and $\sigma\in[0,1]$, equation \eqref{p1} becomes
$$\left\{\begin{array}{ll} 
\ds -\Delta u=\lambda \left( \sigma{ e^{ u}\over \int\limits_\Omega e^{  u}dx} -(1-\sigma)\tau { e^{ -\tau  u}\over \int\limits_\Omega e^{ -\tau   u}dx}
\right)& \hbox{in}\ \Omega\\
 u=0& \hbox{on}\ \partial\Omega,
 \end{array}\right.$$
which can be rewritten (setting $\lambda_1=\lambda\sigma$, $\lambda_2=\lambda(1-\sigma)$ and $V_1=V_2=1$) as
\begin{equation}\label{p22}
\left\{\begin{array}{ll} 
\ds -\Delta u=\lambda_1  { V_1 e^{ u}\over \int\limits_\Omega V_1 e^{  u}dx} -\lambda_2\tau { V_2 e^{ -\tau  u}\over \int\limits_\Omega V_2 e^{- \tau   u}dx}& \hbox{in}\ \Omega\\
  u=0& \hbox{on}\ \partial\Omega.
  \end{array}\right.
  \end{equation}
If $\tau =1$ and $V_1=V_2\equiv 1$ problem \eqref{p22} reduces to the sinh-Poisson equation or its related mean field version, which has received a considerable interest in recent years, see \cite{bjmr,GP,j,jwy1,jwy2,jwyz,os1,r} and
the references therein. 

\medskip \noindent Up to our knowledge, only few results are  known in a  more general situation. 
In \cite{pr1} Pistoia and Ricciardi built blowing-up solutions   to \eqref{p22} when  $\tau  >0$ and $\lambda_1,\lambda_2\tau^2$ are close to $8\pi$, while in \cite{pr2} the same authors built an arbitrary large number of  sign-changing blowing-up solutions to \eqref{p22} when $\tau  >0$ and $\lambda_1,\lambda_2\tau^2$ are close to suitable (not necessarily integer) multiples of $8\pi.$  In \cite{rt}  Ricciardi and Takahashi provided a complete blow-up picture for solution sequences of \eqref{p22}   and successively  in \cite{rtzz} Ricciardi et al. constructed min-max solutions  when $\lambda_1 \to 8\pi^+$ and   $\lambda_2 \to 0$ on a multiply connected domain (in this case the nonlinearity $e^{-\tau  u}$ 
  may  be treated as a lower-order term with respect to the main term $e^u$). In a compact Riemann surface, a blow-up analysis is performed in \cite{j2,rz} and some existence results are obtained when $\tau>0$.
  
\medskip \noindent A natural question concerns whether do there exist solutions to \eqref{p22} on multiply connected domain $\Omega$ for general values of the parameters $\lambda_1,\lambda_2>0$. For the classical mean field equation \eqref{p2} Ould-Ahmedou and Pistoia \cite{op} proved that on a pierced domain $\Omega_\epsilon:=\Omega\setminus \overline{B(\xi_0,\epsilon)}$, $\xi_0\in\Omega$, there exists a solution to \eqref{p2} which blows-up at $\xi_0$ as $\epsilon \to 0$ for any $\lambda>8\pi$ (extra symmetric conditions are required when $\la \in 8\pi \mathbb{N}$). In the present paper we  consider \eqref{p22} on domains $\Om_\e:=\Om \setminus \cup_{i=1}^m \overline{B(\xi_i,\e_i)}$ with several small holes, where $\xi_1,\dots,\xi_m$ are distinct points in $\Omega$ and $\e=(\e_1,\dots,\e_m)$ is small. The main assumption is that $\lambda_1, \lambda_2$ decompose as
\begin{equation}\label{choice0}
\la_1=4\pi(\al_1+\dots+\al_{m_1}), \quad \la_2 \tau^2 = 4\pi(\al_{m_1+1}+\dots+\al_{m }),\
m_1 \in \{0,1,\dots,m\},\ \al_i> 2,\ \al_i\not\in 2\mathbb N.
\end{equation}
Condition \eqref{choice0} when $m_1=m$ is simply equivalent to have $ \la_1>8\pi m $. In general, for the decomposition \eqref{choice0} to hold for $1\leq m_1<m$ and suitable $\alpha_i$'s a necessary condition is that $\la_1>8\pi m_1$ and $\lambda_2\tau^2>8\pi (m-m_1)$. Our main result reads as follows.
\begin{theorem}\label{main} If  \eqref{choice0}  holds,  there exist radii $\epsilon_1,\dots,\epsilon_m$  small enough such that \eqref{p22} has a solution $u_\e$ in $\Om_\e$ blowing-up positively and negatively at $\xi_1,\dots,\xi_{m_1}$ and $\xi_{m_1+1},\dots,\xi_m$, respectively, as $\epsilon_1,\dots,\epsilon_m \to 0$.
\end{theorem}
\noindent

\medskip \noindent Let us briefly describe how we build the solution $u_\epsilon$ using   a perturbative approach. We look for a solution of \eqref{p22} as
\begin{equation}\label{ansatz}
u_\e=P_\e U +\phi_\e,
\end{equation}
where $U$ is a suitable ansatz, $P_\e$ is the projection operator onto $H^1_0(\Omega_\e)$(see \eqref{ePu}) and $\phi_\e\in H_0^1(\Om_\e)$ is a small remainder term. The ansatz $U$ is built as follows. Letting
$$w=\log\frac{2\alpha^2\delta^{\alpha}}{(\delta^{\alpha}+|x -\xi|^{\alpha})^2}$$
be a solution of the singular Liouville equation
\begin{equation*}
\left\{ \begin{array}{ll}\Delta w+|x -\xi|^{\alpha-2}e^{w}=0 &\text{in $\R^2$}\\
\ds\int_{\R^2} |x -\xi |^{\alpha-2} e^w dx<+\infty,& 
\end{array} \right.
\end{equation*}
denote by $U_i$ be the function $w$ corresponding to $\alpha_i,\xi_i$ and $\delta_i>0$, $i=1,\dots,m$. Then $U$ is defined as
$$U= \sum_{k=1}^{m_1} U_k-\frac1\tau\sum_{k=m_1+1}^{m} U_k.$$
In section \ref{sec2} a careful choice of the parameters $\delta_j $'s and the radii $\e_j$'s (see \eqref{choice1}) is needed in order to make $P_\e U$ be a good approximated solution: indeed we will show that the error term $\ml{R}$ given by
\begin{equation}\label{R} 
\ml{R}= \Delta P_\e U+\la_1\dfrac{V_1 e^{P_\e U}}{\int_{\Om_\e}V_1 e^{P_\e U} dx }-\la_2\tau\dfrac{V_2e^{-\tau P_\e U}}{\int_{\Om_\e}V_2 e^{-\tau P_\e U}dx}
\end{equation}
is  small in $L^p$-norm for $p>1$ close to $1$ (see Lemma \ref{erre}). A linearization procedure around $P_\e U$ leads us to re-formulate \eqref{p22} in terms of a nonlinear problem for $\phi_\e$ (see equation \eqref{ephi}). 
Thanks to some estimates in section \ref{sec4} (see \eqref{estlaphi} and \eqref{estnphi}) we will prove the existence of such a solution $\phi_\e$ to \eqref{ephi} by using a fixed point argument. The corresponding solution $u_\e$ in \eqref{ansatz} blows-up at the point $\xi_i$'s thanks to the asymptotic properties of its main order term $P_\e U$ (see Corollary \ref{coro927}). In Section \ref{sec3} we will prove the invertibility of the linear operator naturally associated to the problem (see \eqref{ol}) stated in  Proposition \ref{elle}. Finally, we point out that this approach turns out also useful to address a sinh-Poisson type equation, which is related, but not equivalent to problem \eqref{p22} and it is carried out in \cite{f}.


\section{The ansatz}\label{sec2}
\noindent Let $G(x,y)=-\frac{1}{2\pi}  \log |x-y|+H(x,y)$ be the Green function of $-\Delta$ in $\Omega$, where the regular part $H$ is a harmonic function in $\Omega$ so that $H(x,y)=\frac{1}{2\pi} \log |x-y|$ on $\fr\Om$. Let us introduce the coefficients $\be_{ij},$ $i,j=1,\dots,m,$ as the solution of the linear system
\begin{equation}\label{eqb2}
 {\be_{ij}}\left({1\over 2\pi}\log \e_j -H(\xi_j,\xi_j)\right)-\sum_{k\ne j}\be_{ik} G(\xi_j,\xi_k)=- 4\pi \al_i  H(\xi_i,\xi_j)+\left\{  \begin{array}{ll} 2 \al_i \log \delta_i &\hbox{if }i=j\\
2 \al_i  \log |\xi_i-\xi_j|& \hbox{if } i\not=j. \end{array} \right.
\end{equation}
Notice that \eqref{eqb2} can be re-written as the diagonally-dominant system
$${\be_{ij}} \log \e_j -2\pi \bigg[ \be_{ij} H(\xi_j,\xi_j)+ \sum_{k\ne j}\be_{ik} G(\xi_j,\xi_k) \bigg]=-8 \pi^2 \al_i H(\xi_i,\xi_j)+\left\{  \begin{array}{ll} 4\pi \al_i \log \delta_i &\hbox{if }i=j\\
4\pi \al_i  \log |\xi_i-\xi_j| & \hbox{if } i\not=j \end{array} \right.$$
for $\e_j$ small, which has a unique solution satisfying 
\begin{equation}\label{eqbij}
\beta_{ij}=\frac{4\pi \al_i \log \delta_i}{\log \e_j} \delta_{ij}+O(|\log \e_j|^{-1})
\end{equation}
where $\delta_{ij}$ is the Kronecker symbol. Introducing the projection $P_\e w$ as the unique solution of
\begin{equation}\label{ePu}
\left\{ \begin{array}{ll} 
\Delta P_\e w =\lap w &\text{in }\Om_\e\\
P_\e w=0,&\text{on }\fr\Om_\e,
\end{array}\right.
\end{equation} 
we have the following asymptotic expansion of $P_{\e}U_i$:
\begin{lemma}\label{ewfxi}
There hold
\begin{eqnarray}\label{pui}
P_\e U_i = U_i-\log\lf[2 \al_i^2\de_i^{\al_i}\rg]+
4 \pi \al_i H(x,\xi_i)-\sum_{k=1}^m\be_{ik} G(x,\xi_k)+O \left( \de_i^{\al_i}+\Big(1+\frac{\log \delta_i }{\log \e_i}\Big)  \sum_{k=1}^m\e_k +\Big(\dfrac{ \e_i}{\de_i }\Big)^{\al_i }\right)
\end{eqnarray}
uniformly in $\Om_\e$ and 
\begin{eqnarray} \label{puii}
P_\e U_i=4 \pi \al_i G(x,\xi_i)-\sum_{k=1}^m\be_{ik} G(x,\xi_k)+O \left(\de_i^{\al_i}+\Big(1+\frac{\log \delta_i }{\log \e_i}\Big)  \sum_{k=1}^m\e_k +\Big(\dfrac{ \e_i}{\de_i }\Big)^{\al_i } \right)
\end{eqnarray}
locally uniformly in $\overline{\Om} \sm\{\xi_1,\dots,\xi_m\}$.
\end{lemma}

\begin{proof}[\dem]The harmonic function
$$\psi=P_\e U_i-U_i+\log\lf[2\al_i^2\de_i^{\al_i} \rg]-4\pi \al_i H(x,\xi_i)+\sum_{k=1}^m\be_{ik} G(x,\xi_k)$$
satisfies $\psi=2\log(\de_i^{\al_i}+|x-\xi_i|^{\al_i })-4 \pi \al_i H(x,\xi_i)=O(\de_i^{\al_i})$ on $\fr\Om$ and
\begin{eqnarray*}
&&\hspace{-0.5cm} \psi= 2\log(\de_i^{\al_i}+\e_i^{ \al_i })-4\pi \al_i  H(x,\xi_i)+\displaystyle \sum_{k=1}^m \be_{ik} G(x,\xi_k)= O\lf(\Big(\frac{\e_i}{\delta_i}\Big)^{\alpha_i}+\Big(1+\frac{\log \delta_i }{\log \e_i}\Big)\e_i \rg) \hspace{0.5 cm}  \hbox{on }\fr B(\xi_i,\e_i)\\
&&\hspace{-0.5cm} \psi= 2\log(\de_i^{\al_i}+|x-\xi_i|^{ \al_i })-4\pi \al_i  H(x,\xi_i)+\displaystyle \sum_{k=1}^m \be_{ik} G(x,\xi_k)= O\lf(\de_i^{\al_i}+\Big(1+\frac{\log \delta_i }{\log \e_i}\Big)\e_j \rg) \hbox{ on } \fr B(\xi_j,\e_j)
\end{eqnarray*}
for all $j \not= i$ in view of \eqref{eqb2}-\eqref{eqbij}. By the maximum principle we conclude the validity of \eqref{pui}, and then \eqref{puii} easily follows.
\end{proof}
\noindent Notice that by \eqref{pui}-\eqref{puii} $P_\e U$ displays in the expansion near $\xi_i$, $i=1,\dots,m_1$ a term
$$U_i-\lf(\sum_{j=1}^{m_1} \beta_{ji}-\frac{1}{\tau} \sum_{j=m_1+1}^m \beta_{ji}\rg) G(x,\xi_i).$$
Since $-\Delta P_\e U_i=|x-\xi_i|^{\alpha_i-2} e^{U_i}$ needs to match with $e^{P_\e U}$ if $i=1,\dots,m_1$ and $e^{-\tau P_\e U}$ if $i=m_1+1,\dots,m$, we need to impose 
\begin{equation}\label{choice0.1}\left\{\begin{array}{ll}
\sum\limits_{j=1 }^{m_1} \beta_{ji} - \frac1\tau\sum\limits_{j=m_1+1 }^{m } \beta_{ji}=2\pi(\al_i-2) & i=1,\dots,m_1\\
 - \tau\sum\limits_{j=1 }^{m_1} \beta_{ji}+ \sum\limits_{j=m_1+1 }^{m } \beta_{ji}=2\pi(\al_i-2) & i=m_1+1,\dots, m.
\end{array} \right.
\end{equation}
Thanks to \eqref{eqbij}, \eqref{choice0.1} requires at main order that $\alpha_i \log \delta_i=\frac{\alpha_i-2}{2} \log \e_i$, i.e. $\delta_i^{\alpha_i}\sim \e_i^{\frac{\alpha_i-2}{2}}$. Moreover, due to the presence of 
$\log\lf[2 \al_i^2\de_i^{\al_i}\rg]$ in \eqref{pui}-\eqref{puii} we need further to assume that the $\de_i^{\al_i}$'s have the same rate, as it is well known in problems  of mean-field form, see for instance \cite{CL1,CL2,DEFM,EsFi}.

\medskip \noindent Summarizing, for any $i=1,\dots,m$ we choose
\begin{equation}\label{choice1}
\de_i^{\al_i}=d_i\e, \quad  \e_i^{\al_i-2\over 2}=r_i\e,
\end{equation}
for a small parameter $\e>0$, where $d_i,r_i$ will be specified below, and introduce 
$$\rho_i= \left\{\begin{array}{ll}
 (\al_i+2)H(\xi_i,\xi_i)+\sum\limits_{j=1\atop j\not=i }^{m_1}   (\al_j+2)G(\xi_i,\xi_j) - \frac{1}{\tau} \sum\limits_{j=m_1+1}^{m }   (\al_j+2)G(\xi_i,\xi_j) & i=1,\dots,m_1\\
(\al_i+2)H(\xi_i,\xi_i) -  \tau \sum\limits_{j=1 }^{m_1}   (\al_j+2)G(\xi_i,\xi_j)+  \sum\limits_{j=m_1+1\atop j\not=i}^{m }   (\al_j+2)G(\xi_i,\xi_j) & i=m_1+1,\dots, m. \end{array} \right.$$
Setting $A_i=\overline{B(\xi_i,\eta)} \setminus B(\xi_i,\e_i)$ for $\eta<\frac{1}{2} \min\{|\xi_i-\xi_j|:i\not= j\}$, by Lemma \ref{ewfxi} we deduce the following expansion.
\begin{corollary} \label{coro927}
Assume the validity of \eqref{choice0.1}. There hold
\begin{eqnarray} \label{1517}
P_\e U = U_i-\log\lf[2 \al_i^2\de_i^{\al_i}\rg]+
(\alpha_i-2) \log |x-\xi_i|+2 \pi \rho_i+O \left( \e+ \sum_{k=1}^m\e^{\frac{2}{\alpha_k-2}} +|x-\xi_i|\right)
\end{eqnarray}
uniformly in $A_i$, $i=1,\dots,m_1$, 
\begin{eqnarray} \label{1518}
-\tau P_\e U_= U_i- \log\lf[2 \al_i^2\de_i^{\al_i}\rg]+(\alpha_i-2) \log |x-\xi_i|+
2 \pi \rho_i+ O \left( \e+ \sum_{k=1}^m\e^{\frac{2}{\alpha_k-2}} +|x-\xi_i|\right)
\end{eqnarray}
uniformly in $A_i$, $i=m_1+1,\dots,m$, and 
\begin{eqnarray} \label{1519}
&& P_\e U=2 \pi \sum_{i=1}^{m_1} (\al_i+2) G(x,\xi_i)-\frac{2 \pi}{\tau} \sum_{i=m_1+1}^m (\al_i+2) G(x,\xi_i)+ O \left( \e+ \sum_{k=1}^m\e^{\frac{2}{\alpha_k-2}}  \right)
\end{eqnarray}
locally uniformly in $\overline{\Om} \sm\{\xi_1,\dots,\xi_m\}$.
\end{corollary}
\noindent In order to achieve the validity of \eqref{choice0.1}, we will make a suitable choice of $r_i$ and $d_i$, as expressed by the following Lemma.
\begin{lemma}\label{lem909} If $r_i=d_i e^{-\pi \rho_i}$ for all $i=1,\dots,m$, then \eqref{choice0.1} does hold.
\end{lemma}
\begin{proof}[\dem] Set
$$\beta_i= \left\{\begin{array}{ll}
\sum\limits_{j=1 }^{m_1} \beta_{ji} - \frac1\tau\sum\limits_{j=m_1+1 }^{m } \beta_{ji} & i=1,\dots,m_1\\
 - \tau\sum\limits_{j=1 }^{m_1} \beta_{ji}+ \sum\limits_{j=m_1+1 }^{m } \beta_{ji} & i=m_1+1,\dots, m.
\end{array} \right.$$
When $j=1,\dots,m_1$ let us add \eqref{eqb2} for $i=1,\dots,m_1$ and $-\frac1\tau$ $\times$ \eqref{eqb2} for $i=m_1+1,\dots,m$ to get
$$-2\al_j \log \de_j+\frac{\beta_j}{2\pi} \log \e_j+(4\pi\al_j-\beta_j) H(\xi_j,\xi_j)+ \sum\limits_{i=1 \atop i \not=j }^{m_1}(4\pi\al_i-\beta_i)G(\xi_i,\xi_j) - \frac1\tau \sum\limits_{i=m_1+1}^{m}(4\pi\al_i -\beta_i )G(\xi_i,\xi_j)=0.$$
Similarly, when $j=m_1+1,\dots,m$ we add $-\tau$ $\times$ \eqref{eqb2} for $i=1,\dots,m_1$ and \eqref{eqb2} for $i=m_1+1,\dots,m$ to get  
$$-2\al_j \log \de_j+\frac{\beta_j}{2\pi} \log \e_j+(4\pi\al_j-\beta_j)H(\xi_j,\xi_j) - \tau \sum\limits_{i=1}^{m_1} (4\pi\al_i-\beta_i)G(\xi_i,\xi_j)+ \sum\limits_{i=m_1+1 \atop i\not=j }^{m}(4\pi\al_i-\beta_i) G(\xi_i,\xi_j)=0.$$
Since
$$-2\alpha_j \log \delta_j+\frac{\beta_j}{2\pi} \log \e_j=\frac{\beta_j-2\pi(\alpha_j-2)}{\pi(\alpha_j-2)} \log \e -2\log d_j+\frac{\beta_j}{\pi(\alpha_j-2)}\log r_j$$
in view of \eqref{choice1}, the previous conditions form a system of $m$ equations in $\beta_1,\dots,\beta_m$ which has diagonally-dominant form $\beta_j-2\pi(\alpha_j-2)+O(\frac{1}{|\log \e|})=0$ for $\e$ small. The solution $\beta_1,\dots,\beta_m$ is then uniquely determined and we want to check that $\beta_j=2\pi(\alpha_j-2)$. Inserting $\beta_j=2\pi(\alpha_j-2)$ into the system, it reduces to
$$\log \frac{r_j}{d_j}+\pi \rho_j=0\qquad j=1,\dots,m,$$
which is always true by the choice of $r_i$ and $d_i$. \end{proof}
\noindent Finally, we need to impose that $V_1 e^{PU_\e}$ and $V_2 e^{-\tau PU_\e}$ give integral contributions on the balls $B(\xi_i, \delta)$ for $i=1,\dots,m_1$ and for $i=m_1+1,\dots,m$, respectively, which are proportional to the $\alpha_j$'s. As we will see below, this is achieved by requiring that 
\begin{equation}\label{choice3}\left\{\begin{array}{ll}
\al_i {  V_1(\xi_j) e^{2\pi \rho_j }\over d_j \al_j} =\al_j {V_1(\xi_i) e^{2\pi \rho_i }\over d_i \al_i} &i,j=1,\dots,m_1\\
\al_i {  V_2(\xi_j) e^{2\pi  \rho_j }\over d_j \al_j} =\al_j {V_2(\xi_i) e^{2\pi  \rho_i }\over d_i \al_i}&i,j=m_1+1,\dots,m.
\end{array}\right.
\end{equation}
The choice 
\begin{equation}\label{choiced}
d_i=\left\{ \begin{array}{ll}
\frac{V_1(\xi_i) e^{2\pi  \rho_i}}{\alpha_i^2} &i=1,\dots,m_1\\
\frac{V_2(\xi_i) e^{2\pi  \rho_i}}{\alpha_i^2} &i=m_1+1,\dots,m
\end{array} , \right. \quad r_i=\left\{ \begin{array}{ll}
\frac{V_1(\xi_i) e^{\pi \rho_i}}{\alpha_i^2} &i=1,\dots,m_1\\
\frac{V_2(\xi_i) e^{\pi \rho_i}}{\alpha_i^2} &i=m_1+1,\dots,m
\end{array} \right.
\end{equation}
guarantees the validity of \eqref{choice0.1} and \eqref{choice3} in view of Lemma \ref{lem909}. We are now ready to estimate the precision of our ansatz $U$.
\begin{lemma}\label{erre}
There exists $\e_0>0,$ $p_0>1$ and  $C>0$ such that for any $\e\in(0,\e_0)$ and $p\in(1,p_0)$
\begin{equation}\label{re1}
\|\ml{R}\|_p\le  C \e^{\sigma_p}
\end{equation}
for some $\sigma_p>0$.
\end{lemma}
\begin{proof}[\dem] Setting $\Lambda=\max\{\alpha_1,\dots,\alpha_m\}$, by \eqref{1517}-\eqref{1518} and the change of variable $x=\delta_i y+\xi_i$ let us estimate 
\begin{eqnarray}
\int_{A_i} V_1 e^{P_\e U} dx&=&
\frac{V_1(\xi_i) e^{2 \pi \rho_i}}{2 \al_i^2\de_i^{\al_i}}\int_{A_i}  |x-\xi_i|^{\alpha_i-2} e^{U_i}\lf[1+ O\bigg( \e+ \sum_{k=1}^m\e^{\frac{2}{\alpha_k-2}} +|x-\xi_i| \bigg)\rg]dx \nonumber\\
&=& \frac{\alpha_i^2}{\e} \int_{\frac{\e_i}{\delta_i}\leq |y|\leq \frac{\eta}{\delta_i} } \frac{|y|^{\alpha_i-2}}{(1+|y|^{\alpha_i})^2 }  \lf[1+  O\bigg( \e+ \sum_{k=1}^m\e^{\frac{2}{\alpha_k-2}} +\delta_i |y| \bigg)\rg] dy \nonumber \\
&=& \frac{2\pi \alpha_i}{\e} \lf[1+O(\e^{\frac{1}{\Lambda}}) \rg] \label{1036}
\end{eqnarray}
for any $i=1,\dots,m_1$ and similarly
\begin{eqnarray} \label{1037}
\int_{A_i} V_2 e^{-\tau P_\e U} dx=\frac{2\pi \alpha_i}{\e} \lf[1+O(\e^{\frac{1}{\Lambda}}) \rg]
\end{eqnarray}
for any $i=m_1+1,\dots,m$, in view of \eqref{choice1}, \eqref{choiced} and
$$\int\limits_{\mathbb R^2}  {|y|^{\alpha_i-2}\over \left(1+|y|^{\al_i}\right)^2}dy={2\pi\over\alpha_i}.$$
By \eqref{1518} we have that
\begin{eqnarray} \label{1503} V_1 e^{P_\e U} =
O\lf(\Big[\frac{|x-\xi_i|^{\alpha_i-2}}{\delta_i^{\alpha_i}} e^{U_i}\Big]^{-\frac{1}{\tau}}\rg)\quad \text{uniformly in $A_i$, for $i=m_1+1,\dots,m$,}  
\end{eqnarray} 
and by \eqref{choice1} and \eqref{1503} we get the estimate
\begin{eqnarray}
\int_{A_i} V_1 e^{P_\e U} dx&=&
O\bigg(\de_i^{\al_i\over \tau} \int_{A_i}  \Big[|x-\xi_i|^{\alpha_i-2} e^{U_i}\Big]^{-{1\over\tau} } dx \bigg) = O\bigg(\de_i^{{\al_i+2 \over \tau} +2} \int\limits_{\frac{\e_i}{\delta_i}\leq |y|\leq \frac{\eta}{\delta_i} } \Big[\frac{|y|^{\alpha_i-2}}{(1+|y|^{\alpha_i})^2 }  \Big]^{-{1\over \tau} } dy\bigg)\nonumber\\
&=&  O\bigg(\de_i^{{\al_i +2 \over \tau} +2} \bigg[\int_{\frac{\e_i}{\delta_i}}^1  s^{1-{\al_i-2\over\tau} }\,ds + \int_1^\frac{\eta}{\delta_i}  s^{1+{\al_i+2\over\tau} }\,ds\bigg]\bigg) = O(1) \label{0054}
\end{eqnarray}
for all $i=m_1+1,\dots,m$. Similarly, by \eqref{choice1}-\eqref{1517} we deduce that
\begin{eqnarray} \label{0107}
\int_{A_i} V_2 e^{-\tau P_\e U} dx=O(1)
\end{eqnarray}
for $i=1,\dots,m_1$, in view of
\begin{eqnarray}\label{1502}
V_2 e^{-\tau P_\e U} =
O\lf(\Big[\frac{|x-\xi_i|^{\alpha_i-2}}{\delta_i^{\alpha_i}} e^{U_i}\Big]^{-\tau}\rg) \quad \text{ uniformly in $A_i$, for $i=1,\dots,m_1$ .}
\end{eqnarray}
Therefore, by using \eqref{1519}, \eqref{1036}-\eqref{1037} and \eqref{0054}-\eqref{0107} we deduce that
\begin{eqnarray} \label{ex3}
&& \hspace{-0.4cm}\int_{\Omega_\e} V_1 e^{ PU_\e} dx =\sum_{i=1}^{m _1} \int_{A_i} V_1 e^{P_\e U} dx  +\sum\limits_{i=m_1+1}^{m }\int\limits_{A_i} V_1 e^{P_\e U} dx  +O(1)
={\lambda_1 \over 2 \e}[1  +O(\e^\sigma)]\\
&& \hspace{-0.4cm} \int_{\Omega_\e} V_2 e^{ -\tau PU_\e} dx =\sum_{i=1}^{m _1} \int_{A_i} V_2 e^{-\tau P_\e U} dx  +\sum\limits_{i=m_1+1}^{m }\int\limits_{A_i} V_2 e^{-\tau P_\e U} dx  +O(1)
={\lambda_2 \tau^2  \over 2 \e}[1 +O(\e^\sigma)] \label{1523}
\end{eqnarray}
 in view of \eqref{choice0}, where $\sigma=\frac{1}{\Lambda}$. 
 
\medskip \noindent Since
$$\Delta P_\e U=\left\{ \begin{array}{ll} -|x-\xi_i|^{\alpha_i-2} e^{U_i}+O(\e)&\hbox{in }A_i, \ i=1,\dots,m,\\ 
\frac{ 1}{\tau} |x-\xi_i|^{\alpha_i-2}e^{U_i}+O(\e)& \hbox{in }A_i,\ i=m_1+1,\dots,m\\
O(\e) &\hbox{in }\Omega_\e \setminus \ds \bigcup_{i=1}^m A_i \end{array} \right.$$
in view of  \eqref{choice1}, by \eqref{1517}-\eqref{1518} and \eqref{1502}-\eqref{1523} we can estimate the error term $\ml{R}$ as:
 \begin{eqnarray} \label{ex5}
\ml{R} =|x-\xi_i|^{\al_i-2}e^{U_i} O\Big(\e^\sigma+|x-\xi_i|\Big) +O(\e^\sigma) 
\end{eqnarray}
in $A_i$, $i=1,\dots,m_1$, and 
\begin{eqnarray} \label{ex6.1}
\ml {R} =- \frac{1}{\tau}  |x-\xi_i|^{\alpha_i-2} e^{U_i} O\Big(\e^\sigma+|x-\xi_i|\Big)+O(\e^\sigma) 
\end{eqnarray}
in $A_i$, $i=m_1+1,\dots,m$, while $ \ml {R}=O(\e)$ does hold in $\Omega_\e \setminus \ds \bigcup_{i=1}^m A_i $. By \eqref{ex5}-\eqref{ex6.1} we finally get that there exist $\e_0>0$ small, $p_0>1$ close to $1$ so that $\|\ml {R}\|_p=O(\e^{\sigma_p})$ for all $0<\e\leq \e_0$ and $1<p\leq p_0$, for some $\sigma_p>0$.
\end{proof}


\section{The nonlinear problem and proof of main result}\label{sec4}

In this section we shall study the following nonlinear problem:
\begin{equation}\label{ephi}
\left\{ \begin{array}{ll}
\mathcal L(\phi)= -[\mathcal R+\Lambda(\phi)+\mathcal N(\phi)] & \text{in } \Om_\e\\
\phi=0, &\text{on }\fr\Om_\e,
\end{array} \right.
\end{equation}
where the linear operators $\mathcal L,\Lambda$ are defined as
\begin{equation}\label{ol}
\mathcal L(\phi) = \Delta \phi + K_1\lf(\phi - {1\over \la_1}\int_{\Om_\e} K_1 \phi  dx \rg) + K_2 \lf(\phi - {1\over \la_2\tau^2} \int_{\Om_\e} K_2\phi dx \rg)
\end{equation}
and
\begin{equation}\label{ola}
\begin{split}
\Lambda (\phi) =&\,  \la_1 {V_1 e^{P_\e U}\over\int_{\Om_\e} e^{P_\e U} dx}\lf(\phi - {\int_{\Om_\e} V_1e^{P_\e U}\phi dx \over\int_{\Om_\e} V_1e^{P_\e U} dx } \rg)+\la_2\tau^2 {V_2e^{-\tau P_\e U}\over\int_{\Om_\e} e^{-\tau P_\e U}dx }\lf(\phi - {\int_{\Om_\e} V_2 e^{-\tau P_\e U}\phi dx \over\int_{\Om_\e} V_2(x)e^{-\tau P_\e U}dx} \rg)\\
&\, - K_1\lf(\phi - {1\over \la_1}\int_{\Om_\e} K_1 \phi  dx \rg)-  K_2 \lf(\phi - {1\over \la_2\tau^2} \int_{\Om_\e} K_2\phi dx \rg)
\end{split}
\end{equation}
with
\begin{equation}\label{K12}
K_1=\sum_{k=1}^{m_1}|x-\xi_k|^{\al_k-2}e^{U_k}, \quad K_2=\sum_{k=m_1+1}^{m}|x-\xi_k|^{\al_k-2}e^{U_k}.
\end{equation}
The nonlinear term $\mathcal N(\phi) $ is given by
\begin{equation}\label{nlt}
\begin{aligned}
\mathcal N(\phi)= &\la_1 \lf[{V_1e^{P_\e U+\phi}\over\int_{\Om_\e}
V_1e^{P_\e U+\phi} dx}-{V_1 e^{P_\e U}\over\int_{\Om_\e}V_1e^{P_\e U} dx}-{V_1e^{P_\e U}\over\int_{\Om_\e}V_1
e^{P_\e U} dx}\lf(\phi - {\int_{\Om_\e} V_1e^{P_\e U}\phi dx \over\int_{\Om_\e}V_1
e^{P_\e U} dx} \rg)\rg]\\ & - \la_2\tau \lf[{V_2 e^{-\tau( P_\e U+\phi) }dx \over\int_{\Om_\e}
V_2 e^{-\tau (P_\e U+\phi)}dx}-{V_2 e^{-\tau P_\e U}\over\int_{\Om_\e} V_2 e^{-\tau P_\e U} dx } +\tau{V_2 e^{ - \tau P_\e U}\over\int_{\Om_\e}V_2
e^{-\tau P_\e U} dx }\lf(\phi - {\int_{\Om_\e} V_2e^{ - \tau P_\e U}\phi dx \over\int_{\Om_\e} V_2
e^{-\tau P_\e U} dx } \rg)\rg].
\end{aligned}
\end{equation}
It is readily checked that $\phi$ is a solution to \eqref{ephi} if and only if $u_\e$ given by \eqref{ansatz} is a solution to \eqref{p22}. In section \ref{sec3} we will prove the following result.

\begin{proposition}\label{elle}
For any $p>1,$ there exists $\e_0>0 $ and  $C>0$ such that for any $\e\in(0,\e_0)$ and $h \in L^p(\Om_\e)$ there exists a unique $\phi\in H^1_0(\Om_\e)$ solution of
\begin{equation}\label{pl}
\mathcal L(\phi)=h \ \hbox{ in }\ \Om_\e,\ \ \ \phi=0\ \hbox{ on }\ \partial\Omega_\e,
\end{equation}
which satisfies
\begin{equation}\label{estphi}
\|\phi\|\le C|\log\e|\ \|h\|_p.
\end{equation}
\end{proposition}

We are now in position to study the nonlinear problem \eqref{ephi} and to prove our main result Theorem \ref{main}.

\begin{proposition}\label{p3}
There exist $p_0>1$ and $\e_0>0$ so that for any $1<p<p_0$ and 
all $0<\e\leq \e_0$, the problem \eqref{ephi} admits a unique solution $\phi(\e) \in H_0^1(\Om_\e)$, where $\mathcal L$, $\mathcal R$, $\Lambda(\phi)$ and $\mathcal N$ are given by \eqref{ol}, \eqref{R}, \eqref{ola} and \eqref{nlt}, respectively. Moreover, there is a constant $C>0$ such that
$$\|\phi\|_\infty\le C\e^{\sigma_p} |\log \e |,$$
for some $\sigma_p>0$.
\end{proposition}

Here, $\sigma_p$ is the same as in \eqref{re1}. We shall use the following estimates.

\begin{lemma}
There exist $p_0>1$ and $\e_0>0$ so that for any $1<p<p_0$ and 
all $0<\e\leq \e_0$ it holds
\begin{equation}\label{estlaphi}
\| \Lambda (\phi) \|_p  \le C\e^{\sigma_p'} \|\phi\|,
\end{equation}
for all $\phi\in H_0^1(\Om_\e)$ with $\|\phi\|\le \nu \e^{\sigma_p}|\log\e|$, for some $\sigma_p'>0$.
\end{lemma}

\begin{proof}[\dem]
For simplicity, we denote $W_i= \ds{ \la_i\tau^{2(i-1)} V_i(x) e^{(-\tau)^{i-1} P_\e U}\over \int_{\Om_\e} V_i(x) e^{(-\tau)^{i-1} P_\e U } dx}$ for $i=1,2$. By using \eqref{1517}-\eqref{1519}, \eqref{ex3}-\eqref{1523} and similar computations as to obtain \eqref{ex5}-\eqref{ex6.1}, we find that
$$W_{1}(x)=|x-\xi_i|^{\al_i-2}e^{U_i} \lf[1+O(|x-\xi_i|+\e^{\ti\sigma_1})\rg]$$
uniformly for $x\in A_i$, $i=1,\dots,m_1$, $W_1(x)=O(\e)$ uniformly for $\ds x\in \Om\sm \bigcup_{i=1}^{m_1} A_i$ and
$$W_{2}(x )=|x - \xi_i |^{\al_i-2}e^{U_i} \lf[1+O(|x-\xi_i|+\e^{\ti\sigma_2}) \rg] $$
uniformly for $x\in A_i$, $i=m_1+1,\dots, m$ and $W_2(x)=O(\e)$ uniformly for $\ds x\in \Om\sm \bigcup_{i=m_1+1}^{m} A_i$. Also, from the definition of $K_1$ and $K_2$ in \eqref{K12} it follows that $\ds K_1=\sum_{i=1}^{m_1} O\lf(\de_i^{\al_i}\rg) =O(\e)$ uniformly for $\ds x\in\Om\sm \bigcup_{i=1}^{m_1} A_i$ and $\ds K_2=\sum_{i=m_1+1}^{m} O\lf(\de_i^{\al_i}\rg)=O(\e)$ uniformly for $\ds x\in\Om\sm \bigcup_{i=m_1+1}^{m} A_i $. Hence, for any $q\ge 1$ there holds
\begin{equation*}
\begin{split}
\lf\|W_1- K_1\rg\|_q^q&\,\le C\bigg[\sum_{i=1}^{m_1} \int_{ A_i} \lf(|x-\xi_i |^{\al_i-2}e^{ U_i }[|x-\xi_i| +\e^{\ti\sigma_1 }] \rg)^q+\int_{\Om\sm \cup_{i=1}^{m_1} A_i} \lf( | W_1|^q + |K_1|^q\rg)\bigg] \\
&\,\le C\bigg[\sum_{i=1}^{m_1} \bigg( \de_i^{2-q} \int_{ A_i-\xi_i\over \de_i } \bigg| {2\al_i^2 |y |^{\al_i-1} \over (1+|y|^{\al_i} )^2 }\bigg|^q\, dy  +\e^{\ti \sigma_1 q }  \int_{ A_i-\xi_i\over \de_i } \bigg| {2\al_i^2 |y |^{\al_i-2} \over (1+|y|^{\al_i} )^2 }\bigg|^q\, dy  \bigg)+\e^q\bigg]\\
&\,\le C \e^{ q \sigma'_{1,q} } 
\end{split}
\end{equation*}
for some $\sigma'_{1,q}$. Similarly, we find that
\begin{equation*}
\begin{split}
\lf\|W_2- K_2\rg\|_q^q
&\,\le C\bigg[\sum_{i=m_1+1}^{m} \bigg( \de_i^{2-q} \int_{ A_i-\xi_i\over \de_i } \bigg| {2\al_i^2 |y |^{\al_i-1} \over (1+|y|^{\al_i} )^2 }\bigg|^q\, dy  +\e^{\ti\sigma_2 q }  \int_{ A_i-\xi_i\over \de_i } \bigg| {2\al_i^2 |y |^{\al_i-2} \over (1+|y|^{\al_i} )^2 }\bigg|^q\, dy \bigg)+\e^q \bigg] \\
&\, \le C \e^{q\sigma'_{2,q} } 
\end{split}
\end{equation*}
for some $\sigma'_{2,q}$. It is possible to see that taking $q>1$ close enough to 1, we get that $\sigma'_{i,q}>0$ for $i=1,2$. 

\medskip \noindent Notice that $\Lambda$ is a linear operator and we re-write $\Lambda(\phi)$ as
\begin{equation*}
\begin{split}
\Lambda(\phi)=&\,\sum_{i=1}^2\bigg[ \lf(W_i-K_i\rg)\phi-{1\over\la_i\tau^{2(i-1)} } \lf(W_i-K_i \rg) \int_{\Om_\e}W_i\phi+ {1\over\la_i \tau^{2(i-1)} } K_i \int_{\Om_\e} \lf(K_i-W_i\rg)\phi
 \bigg].
\end{split}
\end{equation*}
Hence, we get that
\begin{equation*}
\begin{split}
\| \Lambda (\phi) \|_p  \le &\,\sum_{i=1}^2\bigg[ \lf\|\lf(W_i-K_i\rg)\phi \rg\|_p+{1\over\la_i \tau^{2(i-1)}} \lf \|\lf(W_i-K_i \rg) \int_{\Om_\e}W_i\phi \rg\|_p + {1\over\la_i \tau^{2(i-1)}} \lf\| K_i \int_{\Om_\e} \lf(K_i-W_i\rg)\phi\rg\|_p
 \bigg]\\
\le &\, \sum_{i=1}^2\bigg[\lf\| W_i-K_i\rg\|_{pr_{i0}} \|\phi \|_{ps_{i0} }+{1\over\la_i \tau^{2(i-1)} } \lf \| W_i-K_i \rg\|_p \|W_i\|_{r_{i1}}  \|\phi\|_{s_{i1} } \\
&\qquad+ {1\over\la_i \tau^{2(i-1)} } \| K_i\|_p \lf\| K_i-W_i\rg\|_{r_{i2}} \|\phi \|_{s_{i2}} \bigg]\\
\le &\, C\sum_{i=1}^2\bigg[\e^{\sigma'_{i,pr_{i0} } }  \|\phi \| + \e^{\sigma'_{i,p} +\sigma_{3,r_{i1}  }}  \|\phi\| + \e^{\sigma'_{i,r_{i2} } +\sigma_{3,p} }  \|\phi \| \bigg]\\
\le &\, C \e^{\sigma'_{p } }  \|\phi \|,
\end{split}
\end{equation*}
where $\ds \sigma'_p=\min\lf\{\sigma'_{i,pr_{i0} }; \sigma'_{i,p} +\sigma_{3,r_{i1}  } ;  \sigma'_{i,r_{i2}} +\sigma_{3,p  } \mid i=1,2; j=1,\dots,m_1\rg\}$ with $ r_{ij} $, $s_{ij} $, $i=1,2$, $j=0,1,2$ satisfying $\dfrac{1}{ r_{ij} }+\dfrac{1}{ s_{ij} }=1$. We have used that
\begin{equation*}
\begin{split}
\|W_1\|_{r_{11}}^{r_{11} } &\le C\lf[ \sum_{j=1}^{m_1} \de_j^{2-2r_{11} }\int_{A_j-\xi_j\over \de_j} \lf| {2\al_j^2|y|^{\al_j -2}\over (1+|y|^{\al_i})^2 } \rg|^{r_{11} }\, dy + \e^{ r_{11} }\rg] \le C\lf[\sum_{i=1}^{m_1}\e^{2- 2r_{11}\over \al_j } + \e^{r_{11}}\rg]\le C \e^{ r_{11} \sigma_{3,r_{11} }} 
\end{split}
\end{equation*}
and
$$\|W_2\|_{r_{21}}^{r_{21} } \le C\lf[ \sum_{j=m_1+1}^{m} \de_j^{2- 2r_{21} }\int_{A_j-\xi_j\over \de_j} \lf| {2\al_j^2|y|^{\al_j -2}\over (1+|y|^{\al_i})^2 } \rg|^{r_{21} }\, dy + \e^{ r_{21} }\rg] \le C \e^{ r_{21} \sigma_{3,r_{21} }},
$$
where for $q>1$ we denote $\ds \sigma_{3, q } = \min\lf\{ {2- 2  q\over \al_j q }\ :\ j=1,\dots,m\rg\}$, and similarly that $\| K_i\|_p^p\le C\e^{p\sigma_{3,p} }$. Note that 
$\ds {2- 2  q\over\al_j q}<1$ for any $j=1,\dots,m$. 
Furthermore, we have used the H\" older's inequality $\|uv\|_q\le \|u\|_{qr}\|v\|_{qs}$ with $\ds{1\over r }+{1\over s }=1$ and the inclusions $L^{p}(\Om_\e)\hookrightarrow L^{pr}(\Om_\e)$  for any $r>1$ and $H^{1}_0(\Om_\e)\hookrightarrow L^{q}(\Om_\e)$ for any $q>1$. Let us stress that we can choose $p$, $r_{ij}$ and $s_{ij}$, $i=1,2$, $j=0,1,2$, close enough to 1 such that $\sigma_p'>0$.
\end{proof}

\medskip

\begin{lemma}
There exist $p_0>1$ and $\e_0>0$ so that for any $1<p<p_0$ and 
all $0<\e\leq \e_0$ it holds
\begin{equation}\label{estnphi}
\| \mathcal N (\phi_1)- \ml N(\phi_2) \|_p  \le C\e^{\sigma_p''} \|\phi_1-\phi_2\|
\end{equation}
for all $\phi_i\in H_0^1(\Om_\e)$ with $\|\phi_i\|\le \nu \e^{\sigma_p}|\log\e|$, $i=1,2$, and for some $\sigma_p''>0$. In particular, we have that
\begin{equation}\label{estnphi1}
\| \ml N (\phi) \|_p  \le C\e^{\sigma_p''} \ \|\phi\|
\end{equation}
for all $\phi\in H_0^1(\Om_\e)$ with $\|\phi\|\le \nu \e^{\sigma_p}|\log\e|$.
\end{lemma}

\begin{proof}[\dem] We will argue in the same way as in \cite[Lemma 5.1]{op}. First, we point out that
\begin{equation*}
\ml N(\phi)=\sum_{i=1}^2\la_i(-\tau)^{i-1}\lf\{f_i(\phi)-f_i(0)-f'_i(0)[\phi] \rg\},\ \ \ \text{where} \ \ \
f_i(\phi)= {V_i(x) e^{(-\tau)^{i-1}(P_\e U+\phi)} \over \int_{\Om_\e} V_i(x)e^{  (-\tau)^{i-1}( P_\e U+\phi)} }.
\end{equation*}
Hence, by the mean value theorem we get that
\begin{equation}\label{mvtn}
\begin{split}
\ml N(\phi_1)-\ml N(\phi_2) &\,=\sum_{i=1}^2 \la_i(-\tau)^{i-1} \lf\{f_i(\phi_1)-f_i(\phi_2)-f'_i(0)[\phi_1-\phi_2] \rg\}\\
&\,=\sum_{i=1}^2\la_i (-\tau)^{i-1} \lf\{f_i'(\phi_{\theta_i} )-f_i'(0) \rg\} [\phi_1-\phi_2] =\sum_{i=1}^2\la_i ( - \tau )^{i-1} f_i''(\ti \phi_{\mu_i}) [\phi_{\theta_i}, \phi_1-\phi_2],
\end{split}
\end{equation}
where $\phi_{\theta_i}=\theta_i\phi_1+(1-\theta_i)\phi_2$, $\ti\phi_{\mu_i}=\mu_i\phi_{\theta_i}$ for some $\theta_i,\mu_i\in[0,1]$, $i=1,2$, and
\begin{equation*}
\begin{split}
f_i''(\phi)[\psi,v]=&\,\tau^{2(i-1)}\bigg[ { V_i(x)e^{ u_i } \psi v\over \int_{\Om_\e} V_i(x)e^{ u_i} }- { V_i(x)e^{u_i } v \int_{\Om_\e} V_i(x)e^{ u_i }\psi \over \big(\int_{\Om_\e} V_i(x)e^{ u_i }\big)^2 }- { V_i(x)e^{ u_i } \psi \int_{\Om_\e} V_i(x)e^{ u_i } v \over \big(\int_{\Om_\e} V_i(x)e^{ u_i } \big)^2 }\\
&\,- { V_i(x)e^{ u_i }  \int_{\Om_\e} V_i(x)e^{ u_i }\psi v\over \big(\int_{\Om_\e} V_i(x)e^{ u_i }\big)^2 } +2{ V_i(x)e^{ u_i } \int_{\Om_\e} V_i(x) e^{ u_i  }v \int_{\Om_\e} V_i(x)e^{ u_i }\psi \over \big(\int_{\Om_\e} V_i(x)e^{ u_i }\big)^3 }\bigg],
\end{split}
\end{equation*}
where for simplicity we denote $ u_i= (-\tau)^{i-1}(P_\e U+\phi)$. Using H\"older's inequalities we get that
\begin{equation}\label{hin}
\begin{split}
\lf\| f_i''(\phi)[\psi,v]\rg\|_p\le&\, |\tau|^{2(i-1)}\bigg[ {\| V_ie^{ u_i } \|_{pr_i}\over \|V_ie^{ u_i}\|_1 }\|\psi\|_{ps_i} \|v\|_{pt_i}+ {\| V_ie^{u_i }\|_{pr_i}^2  \over \|V_ie^{ u_i }\|_1^2 }\|v\|_{pq_i}\|\psi\|_{\ti r_i} + {\| V_ie^{u_i }\|_{pr_i}^2  \over \|V_ie^{ u_i }\|_1^2 }\|\psi \|_{pq_i}\|v\|_{\ti r_i}\\
&\,+ {\| V_ie^{ u_i }\|_p \over \|V_ie^{ u_i }\|_1^2 }  \|V_ie^{ u_i }\|_{pr_i}\|\psi\|_{\ti r_i r_i}\| v\|_{\ti r_i q_i} +2{ \|V_ie^{ u_i }\|_p  \over \|V_ie^{ u_i }\|_1^3 } \|V_ie^{ u_i  }\|_{pr_i}^2 \|v\|_{\ti r_i}\|\psi\|_{\ti r_i}\bigg]\\
\le&\,C \lf[{\| V_ie^{ u_i } \|_{pr_i}\over \|V_ie^{ u_i}\|_1 }+ {\| V_ie^{u_i }\|_{pr_i}^2  \over \|V_ie^{ u_i }\|_1^2 }  +{ \|V_ie^{ u_i }\|_{p r_i}^3  \over \|V_ie^{ u_i }\|_1^3 }  \rg]\|\psi\| \| v\|,
\end{split}
\end{equation}
with $\ds {1\over r_i} +{1\over s_i} + {1\over t_i}=1$, $\ds {1\over r_i}+{1\over q_i}=1$, $\ds {1\over pr_i} + {1\over \ti r_i}=1$. We have used the H\" older's inequality, the inclusions presented in the previous Lemma and $\ds \|uvw\|_q\le \|u\|_{qr}\|v\|_{qs}\|w\|_{qt}$ with $\ds{1\over r }+{1\over s }+{1\over t}=1$. Now, let us estimate $\ds {\| V_ie^{ u_i } \|_{pr_i}\over \| V_i e^{ u_i}\|_1 }$ with $\phi=\ti\phi_{\mu_i}$, $i=1,2$. For $i=1$, arguing exactly as in the proof of \eqref{ex3} we obtain that
$$\lf\| V_1e^{P_\e U  } \rg\|_{q}^q=\sum_{i=1}^{m_1} O\lf( \de_i^{2  - (\al_i+2)q } \rg)=\sum_{i=1}^{m_1} O\lf( \e^{ {2-(\al_i+2)q\over \al_i} }\rg) \quad\text{for any } q\ge 1.$$
Moreover, \eqref{ex3} implies that $\ds \lf\| V_1 e^{ P_\e U } \rg\|_{1}\ge {C_1 \over \e }$ for some $C_1> 0$. For $i=2$, similarly we obtain that
$$\lf\| V_2e^{ -\tau P_\e U } \rg\|_{q}^q=\sum_{i=m_1+1}^{m} O\lf( \e^{ {2-(\al_i+2)q\over \al_i} }\rg)\quad\text{for any } q\ge 1\quad\text{and}\quad \lf\| V_2e^{- \tau P_\e U } \rg\|_{1}\ge {C_2\over \e }\quad\text{for some } C_2> 0.$$
Note that $\ds \sigma_{3,q} -1\le {2-(\al_i+2)q\over \al_i q}\quad\text{ for any} \quad i=1,\dots,m.$
 
On the other hand, using the estimate $|e^a-1|\le |a|$ for any $a\in\R$ we have that
\begin{equation*}
\begin{split}
\lf\| V_ie^{(-\tau)^{i-1}(P_\e U +\ti\phi_{\mu_i})}- V_ie^{(-\tau)^{i-1} P_\e U }\rg\|_q&=\lf(\int_{\Om_\e} \lf| e^{(-\tau)^{i-1}P_\e U} \rg|^q \ \lf| V_ie^{(-\tau)^{i-1}\ti\phi_{\mu_i}}- 1  \rg|^q \rg)^{1/q}\\
&\le \lf(\int_{\Om_\e} \lf| V_ie^{ (-\tau)^{i-1}P_\e U} \rg|^q \ \lf| \ti\phi_{\mu_i} \rg|^q \rg)^{1/q}\\
&\le \lf\| V_ie^{(-\tau)^{i-1}P_\e U }\rg\|_{qs_i}\lf\| \ti\phi_{\mu_i} \rg\|_{qt_i},
\end{split}
\end{equation*}
with $\ds{1\over s_i}+{1\over t_i}=1$, $i=1,2$. Hence, it follows that
$$\lf\| V_ie^{(-\tau)^{i-1}(P_\e U+\ti\phi_{\mu_i} ) }- V_ie^{(-\tau )^{i-1} P_\e U }\rg\|_q\le C \e^{ \sigma_{3,qs_i} -1 } \lf\| \ti\phi_{\mu_i} \rg\|\le C \e^{ \sigma_{3,qs_i} -1}\ \e^{\sigma_p}|\log \e|,$$
in view of $\|\ti\phi_{\mu_i}\|\le\nu  \e^{\sigma_p}|\log \e|$, $i=1,2$. In particular, if $q=1$ we get
$$\lf\| V_ie^{(-\tau)^{i-1}(P_\e U+\ti\phi_{\mu_i} ) }- V_ie^{(-\tau )^{i-1} P_\e U } \rg\|_{1}=O\lf( \e^{ \sigma_{3,s_i} -1+ \sigma_p}|\log \e| \rg)\qquad\text{for any } s_i> 1.$$
By the previous estimates we find that $\ds \lf\| V_i e^{(-\tau)^{i-1} (P_\e U+\ti\phi_{\mu_i})}\rg\|_{q}=O\lf( \e^{ \sigma_{3,qs_i} -1 + \sigma_p}|\log \e|+ \e^{ \sigma_{3,q} -1 } \rg)$.
Also, choosing $s_i$, $i=1,2$, close enough to 1, we get that  $\sigma_p +\sigma_{3,s_i}>0$ and 
$$\lf\| V_ie^{(-\tau)^{i-1}(P_\e U+\ti\phi_{\mu_i} ) }\rg\|_{1}\ge {C_i \over \e } - C \e^{\sigma_{3,s_i} -1+ \sigma_p}|\log \e| \ge { 1 \over \e } \lf(C_i -  C \e^{\sigma_{3,s_i} + \sigma_p }\  |\log \e| \rg)\ge {C_i \over 2 \e } .$$
Taking $q=pr_i$, we obtain the estimate for $i=1,2$
\begin{equation}\label{qie}
\begin{split}
{\| V_ie^{(-\tau)^{i-1}( P_\e U+\ti\phi_{\mu_i}) } \|_{pr_i}\over \| V_i e^{(-\tau)^{i-1}( P_\e  U +\ti\phi_{\mu_i} ) }\|_1 }&=O\lf(  \e\lf[ \e^{\sigma_p+\sigma_{3,pr_is_i}-1 }|\log \e| + \e^{ \sigma_{3,pr_i} -1 } \rg]\rg)\\
&=O\lf(  \e^{\sigma_{3,pr_i} } \Big[ \e^{\sigma_p+\sigma_{3,pr_is_i} - \sigma_{pr_i} } |\log \e|+ 1\Big]     \rg)
=O\lf(  \e^{ \sigma_{3,pr_i} }    \rg)
\end{split}
\end{equation}
choosing $s_i>1$ close enough to 1 so that $\ds \sigma_p +\sigma_{3,pr_is_i} - \sigma_{pr_i}>0$, $i=1,2$. 
Now, we can conclude the estimate by using \eqref{mvtn}-\eqref{qie} to get
\begin{equation*}
\begin{split}
\|\ml N(\phi_1)-\ml N(\phi_2)\|_p &\,\le \sum_{i=1}^2\la_i |\tau|^{i-1}\lf\| f_i''(\ti \phi_{\mu_i}) [\phi_{\theta_i}, \phi_1-\phi_2]\rg\|_p  \le\,C \sum_{i=1}^2\la_i { \|e^{ (- \tau)^{i-1}(P_\e U+\ti\phi_{\mu_i} )}\|_{p r_i}^3  \over \|e^{  (-\tau)^{i-1}(P_\e U+\ti\phi_{\mu_i} )}\|_1^3 }  \| \phi_{\theta_i} \| \| \phi_1-\phi_2\|\\
&\, \le\,C \sum_{i=1}^2\la_i \e^{\sigma_p + 3 \sigma_{3,pr_i} }  | \log \e| \| \phi_1-\phi_2\|
 \le\,C \e^{\sigma_p'' }  \|\phi_1-\phi_2\|,
\end{split}
\end{equation*}
where $\sigma''_p={1\over 2}\min\{ \sigma_p +3 \sigma_{3,pr_i}  \ : \ i=1,2\} >0$ choosing $r_i$ close to 1 so that $\sigma_p+3 \sigma_{3,pr_i} >0$ for $i=1,2$. Let us stress that $p>1$ is chosen so that $\sigma_p>0$.
\end{proof}

\begin{proof}[\bf Proof of the Proposition \ref{p3}] Notice that from Proposition \ref{elle} problem \eqref{ephi} becomes
$$\phi=-T(\ml R+\Lambda(\phi)+\ml N(\phi)):=\ml{A}(\phi).$$
For a given number $\nu>0$, let us consider $
\ml{F}_\nu = \{\phi\in H : \| \phi \| \le
\nu \e^{\sigma_p} |\log\e|\}$. 
From the Proposition \ref{elle}, \eqref{re1}, \eqref{estlaphi} and \eqref{estnphi1}, we get for any $\phi\in \ml{F}_\nu$,
\begin{equation*}
\begin{split}
\|\ml{A}(\phi)\| & \le C | \log \e |\lf[ \|\ml R \|_p+\|\Lambda(\phi)\|_p +
\|\ml N(\phi)\|_p\rg] 
\le C |\log \e |\lf[ \e^{\sigma_p}+\e^{\sigma_p'} \|\phi\|+\e^{\sigma_p''} \|\phi\|\rg]\\
& \le C \e^{\sigma_p} | \log \e |\lf[1 +2 \nu\e^{\min\{\sigma_p',\sigma_p''\}} |\log\e|\rg].
\end{split}
\end{equation*}
Given any $\phi_1,\phi_2\in\ml{F}_\nu$, we have that $\ml{A}(\phi_1)-\ml{A}(\phi_2) = - T\lf(\Lambda(\phi_1-\phi_2)+\ml N(\phi_1)
- \ml N(\phi_2)\rg)$ and
\begin{equation*}
\begin{split}
\|\ml{A}(\phi_1)-\ml{A}(\phi_2)\| & \le C |\log
\e|\lf[\|\Lambda(\phi_1-\phi_2)\|_p+\lf\| \ml N(\phi_1) - \ml N(\phi_2)\rg\|_p\rg]
\le C\e^{\min\{\sigma_p',\sigma_p''\}}|\log\e| \ \|\phi_1-\phi_2\|,
\end{split}
\end{equation*}
with $C$ independent of $\nu$, by using Proposition \ref{elle} and \eqref{estlaphi}-\eqref{estnphi}. Therefore, for some $\sigma>0$ we get that $\|\ml{A}(\phi_1)-\ml{A}(\phi_2)\|  \le C \e^{\sigma} |\log \e| \|\phi_1-\phi_2\|$. It follows that for all $\e$ sufficiently small $\ml{A}$ is a
contraction mapping of $\ml{F}_\nu$ (for $\nu$ large enough), and
therefore a unique fixed point of $\ml{A}$ exists in $\ml{F}_\nu$.

\end{proof}

\begin{proof}[\bf Proof of the Theorem \ref{main}] The existence of a solution 
$$u_\e=\sum_{j=1}^{m_1}P_\e U_j - \dfrac{1}{\tau } \sum_{j=m_1+1}^m P_\e U_j + \phi$$ 
to equation \eqref{p22} follows directly by Proposition \ref{p3}. The asymptotic shape of the solution $u_\e$ as $\e\to 0^+$ follows by the definition of $U_j$, Lemma \ref{ewfxi} and the choice of the parameters \eqref{choice1}-\eqref{choice3}.
\end{proof}


\section{The linear theory} \label{sec3}

In this section we present the invertibility of the linear operator $\mathcal L$ defined in \eqref{ol}. Roughly speaking, in the scale annulus $\de_i^{-1} (B_i-\xi_i)$ the operator $\ml L$ approaches to the following linear operator in $\R^2$
$$L_i(\phi)=\Delta\phi+{2\al_i^2|y|^{\al_i-2}\over (1+|y|^{\al_i})^2}\phi,\qquad i=1,\dots, m.$$
It is well known that the bounded solutions of $L_i(\phi)=0$ in
$\R^2$ are precisely linear combinations of the functions
\begin{equation*}
Y_{1i}(y) = { |y|^{\al_i\over 2} \over 1+|y|^{\al_i}}\cos\Big({\al_i\over 2}\theta\Big),\quad Y_{2i}(y) = { |y|^{\al_i\over 2} \over 1+|y|^{\al_i}}\sin\Big({\al_i\over 2}\theta\Big)\quad\text{and}\quad Y_{0i}(y) = {1-|y|^{\al_i}\over 1+|y|^{\al_i}},
\end{equation*}
which are written in polar coordinates for $i=1,\dots, m$. See \cite{DEM5} for a proof. In our case, we will consider solutions of $L_i(\phi)=0$ such that $\int_{\R^2}|\nabla\phi(y)|^2\, dy<+\infty$, which reduce to multiples of $Y_{0i}$. See \cite[Theorem A.1]{op} for a proof. Another key element in the study of $\ml L$, which shows technical details, is to get rid of the presence of
\begin{equation}\label{ctjphi}
\ti c_{j}(\phi)=-{1\over \la_j\tau^{2(j-1)}}\int_{\Om_{\e_n}} K_j\phi\qquad j=1,2.
\end{equation}

Following ideas presented in \cite{op}, let us introduce the following Banach spaces for $j=1,2$
$$L_{\al_i}(\R^2)=\lf\{u\in W_{\text{loc} }^{1,2}(\R^2)\ :\ \int_{\R^2}{|y|^{\al_i-2}\over (1+|y|^{\al_i})^2}|u(y)|^2\, dy<+\infty\rg\}$$
and
$$H_{\al_i}(\R^2)=\lf\{u\in W_{\text{loc} }^{1,2}(\R^2)\ :\ \int_{\R^2}|\nabla u(y)|^2\, dy+\int_{\R^2}{|y|^{\al_i-2}\over (1+|y|^{\al_i})^2}|u(y)|^2\, dy<+\infty\rg\}$$
endowed with the norms
$$\|u\|_{L_{\al_i}}:=\lf(\int_{\R^2}{|y|^{\al_i-2}\over (1+|y|^{\al_i})^2}|u(y)|^2\, dy\rg)^{1/2}$$
and
$$\|u\|_{H_{\al_i}}:=\lf(\int_{\R^2} |\nabla u(y)|^{2}\, dy+\int_{\R^2}{|y|^{\al_i-2}\over (1+|y|^{\al_i})^2}|u(y)|^2\, dy\rg)^{1/2}.$$
It is important to point out the compactness of the embedding $i_{\al_i}:H_{\al_i}(\R^2)\to L_{\al_i}(\R^2)$ (see for example \cite{GP}).

\begin{proof}[\bf Proof of the Proposition \ref{elle}]
The proof will be done in several steps. Let us assume by contradiction the existence of $p>1$, sequences $\e=\ep_n\to0$ (with a slight abuse of notation), functions $h_n\in L^p(\Om_{\ep_n})$, $\phi_n\in W^{2,2}(\Om_{\ep_n})$ such that
\begin{equation}\label{eqphin}
\mathcal L(\phi_n)=h_n\ \ \text{in}\ \ \Om_{\ep_n},\ \ \phi_n=0\ \ \text{on}\ \ \fr\Om_{\ep_n},
\end{equation}
with $\|\phi_n\|=1$ and $|\log\ep_n|\ \|h_n\|_p=o(1)$ as $n\to+\infty$. We will shall omit the subscript $n$ in $\de_{i,n} =\de_i$. Recall that $\de_i^{\al_i}=d_{i,n}\ep_n$ and points $\xi_1,\dots,\xi_m\in\Om$ are fixed. 

Now, define $\Phi_{i,n}(y):=\phi_{n}(\xi_i+\de_i y)$ for $y\in\Om_{i,n}:=\de_i^{-1}(\Om_{\ep_n}-\xi_i)$, $i=1,\dots,m$. Thus, extending $\phi_n=0$ in $\R^2\sm\Om_{\ep_n}$ we can prove the following fact.

\begin{claim}
The sequence $\{\Phi_{i,n}\}_n$ converges (up to a subsequence) to $\Phi_i^*$ weakly in $H_{\al_i}(\R^2)$ and strongly in $L_{\al_i}(\R^2)$.
\end{claim}

\begin{proof}[\dem]
First, we shall show that the sequence $\{\Phi_{i,n}\}_n$ is bounded in $H_{\al_i}(\R^2)$. Notice that for $i=1,\dots,m$
$$\|\Phi_{i,n}\|_{H^1_0(\Omega_{i,n})}=\int_{\Om_{i,n}}\de_i^2|\grad \phi_{i,n}(\xi_i+\de_i y)|^2\,dy=\int_{\Om_{\ep_n} }|\grad\phi_n(x)|^2\, dx=1.$$ 
Thus, we want to prove that there is a constant $M>0$ such for all $n$ (up to a subsequence)
$$\|\Phi_{i,n}\|_{L_{\al_i}}^2=\int_{\Omega_{i,n} } {|y|^{\al_i-2}\over (1+|y|^{\al_i})^2} \Phi_{i,n}^2(y)\, dy\le M.$$
Notice that for any $i\in\{1,\dots,m\}$ we find that
 in $\Om_{i,n}$
\begin{equation}\label{eqPhi1}
\begin{split}
\lap \Phi_{i,n}&+\de_i^2K_1(\xi_i+\de_i y) \lf(\Phi_{i,n} +\ti c_{1,n} \rg)+\de_i^2K_2(\xi_i+\de_i y) \lf(\Phi_{i,n} +\ti c_{2,n} \rg)=\de_i^2h_n(\xi_i+\de_i y), 
\end{split}
\end{equation}
where for simplicity we denote $\ti c_{j,n}=\ti c_{j}(\phi_n)$, with $\ti c_j$ given by \eqref{ctjphi}. Furthermore, it follows that $\Phi_{i,n}\to\Phi_i^*$ weakly in $H^1_0(\Omega_{i,n})$ and strongly in $L^p(K)$ for any $K$   compact sets in $\mathbb R^2$.
Now, let $\chi$ a smooth function with compact support in $\mathbb R^2.$ We multiply \eqref{eqPhi1} by $\chi$ and we get 
\begin{equation*}
\begin{split}
-\int_{\Omega_{i,n} }&\nabla \Phi_{i,n}\nabla \chi+\int_{\Omega_{i,n} } \de_i^2K_1(\xi_i+\de_i y)  \Phi_{i,n}\chi+\ti c_{1,n}\int_{\Omega_{i,n} }\de_i^2K_1(\xi_i+\de_i y)\chi\\
&+\int_{\Om_{i,n}} \de_i^2K_2(\xi_i+\de_i y)\Phi_{i,n}\chi  +\ti c_{2,n} \int_{\Om_{i,n}} \de_i^2K_2(\xi_i+\de_i y)\chi
=\,\int_{\Omega_{i,n} }\de_i^2h_n(\xi_i+\de_i y) \chi .
\end{split}
\end{equation*}
Hence, we obtain that for $j=1,2$
\begin{equation}\label{dik12}
\de_i^2K_j(\xi_i+\de_i y)=
\begin{cases}
(2-j)\ds{2\al_i^2|y|^{\al_i-2}\over (1+|y|^{\al_i})^2}+O(\de_i^2\e)&\text{ if }i=1,\dots,m_1\\[0.3cm]
(j-1)\ds{2\al_i^2|y|^{\al_i-2}\over (1+|y|^{\al_i})^2} + O(\de_i^2\e)&\text{ if }i=m_1+1,\dots,m
\end{cases}
\end{equation}
uniformly on compact subsets of $\R^2$. Thus, we get that
\begin{equation*}
\begin{split}
-\int_{\Omega_{i,n} }&\nabla \Phi_{i,n}\nabla \chi+\int_{\Omega_{i,n} } \bigg[\ds{2\al_i^2|y|^{\al_i-2}\over (1+|y|^{\al_i})^2}+O(\de_i^2\e)\bigg]\Phi_{i,n}\chi+\ti c_{1,n} \int_{\Omega_{i,n} }\bigg[\ds{2\al_i^2|y|^{\al_i-2}\over (1+|y|^{\al_i})^2}+O(\de_i^2\e)\bigg]\chi\\
&+\int_{\Om_{i,n}} O(\de_i^2\e) \Phi_{i,n}\chi  + \ti c_{2,n} \int_{\Om_{i,n}} O(\de_i^2\e) \chi = \int_{\Omega_{i,n} }\de_i^2h_n(\xi_i+\de_i y) \chi 
\end{split}
\end{equation*}
for $i=1,\dots,m_1$ and 
\begin{equation*}
\begin{split}
-\int_{\Omega_{i,n} }&\nabla \Phi_{i,n}\nabla \chi +\int_{\Om_{i,n}} O(\de_i^2\e) \Phi_{i,n}\chi +\ti c_{1,n}  \int_{\Om_{i,n}} O(\de_i^2\e) \chi +\int_{\Omega_{i,n} } \bigg[\ds{2\al_i^2|y|^{\al_i-2}\over (1+|y|^{\al_i})^2}+O(\de_i^2\e)\bigg]\Phi_{i,n}\chi \\
& + \ti c_{2,n} \int_{\Omega_{i,n} }\bigg[\ds{2\al_i^2|y|^{\al_i-2}\over (1+|y|^{\al_i})^2}+O(\de_i^2\e)\bigg]\chi  
= \int_{\Omega_{i,n} }\de_i^2h_n(\xi_i+\de_i y) \chi 
\end{split}
\end{equation*}
for $i=m_1+1,\dots,m$. We re-write the system for $\ti c_{1,n}$ and $\ti c_{2,n}$ as a diagonal dominant one as $n\to+\infty$
\begin{equation*}
\begin{split}
\ti c_{1,n} \int_{\Omega_{i,n} }\bigg[{2\al_i^2|y|^{\al_i-2}\over (1+|y|^{\al_i})^2} +O(\de_i^2\e)\bigg] \chi+o(1)\ti c_{2,n} =&\,O(1)\\[0.4cm]
o(1)\ \ti c_{1,n}+\ti c_{2,n}\int_{\Omega_{j,n} } \bigg[{2\al_j^2|y|^{\al_j-2}\over (1+|y|^{\al_j})^2} +O(\de_j^2\e)\bigg]  \chi=&\, O(1),
\end{split}
\end{equation*}
choosing $i\in\{1,\dots,m_1\}$ and $j\in\{m_1+1,\dots,m\}$. Thus, if we choose $\chi$ so that $\ds \int_{\mathbb R^2}{2\al_k^2|y|^{\al_k-2}\over (1+|y|^{\al_k})^2}  \chi\ dy\not=0$ for $k=i,j$ then we obtain that $ \ti c_{i,n}=O(1)$,  for $i=1,2$. Now, we multiply \eqref{eqPhi1} by $\Phi_{i,n}$ for any $i\in\{1,\dots,m\}$ and we get 
\begin{equation*}
\begin{split}
-\int_{\Omega_{i,n} }&|\nabla \Phi_{i,n}|^2 + \int_{\Omega_{i,n} } \de_i^2K_1(\xi_i+\de_i y)  \Phi_{i,n}^2+\ti c_{1,n}\int_{\Omega_{i,n} }\de_i^2K_1(\xi_i+\de_i y)\Phi_{i,n}\\
&+\int_{\Om_{i,n}} \de_i^2K_2(\xi_i+\de_i y)\Phi_{i,n}^2  +\ti c_{2,n} \int_{\Om_{i,n}} \de_i^2K_2(\xi_i+\de_i y)\Phi_{i,n}
=\,\int_{\Omega_{i,n} }\de_i^2h_n(\xi_i+\de_i y) \Phi_{i,n}.
\end{split}
\end{equation*}
Hence, we deduce that
\begin{equation}\label{snpsi}
\sum_{i=1}^{m} 2\al_i^2\| \Phi_{i,n}\|_{L_{\al_i}}^2
=\,1 + \la_1 (\ti c_{1,n})^2 + \la_2\tau^2 (\ti c_{2,n})^2 +o(1).
\end{equation}
Therefore, the sequence $\{\Phi_{i,n}\}_n$ is bounded in $H_{\al_i}(\R^2)$, so that there is a subsequence $\{\Phi_{i,n}\}_n$ and functions $\Phi_i^*$, $i=1,2$ such that $\{\Phi_{i,n}\}_n$ converges to $\Phi_i^*$ weakly in $H_{\al_i}(\R^2)$ and strongly in $L_{\al_i}(\R^2)$. That proves our claim.
\end{proof}

Define the sequences $\psi_{i,n}=\phi_n+\ti c_{i,n}$, $i=1,2$. Notice that clearly
\begin{equation}\label{eqpsi}
\lap \psi_{i,n}+K_1 \psi_{1,n}+K_2\psi_{2,n}=h_n\quad\text{in}\ \ \oen,\qquad i=1,2.
\end{equation}
Now, define $\Psi_{i,j,n}(y):=\psi_{i,n}(\xi_j+\de_j y)$ for $y\in \Om_{j,n}$, $i=1,2$ and $j=1,\dots,m$. Note that $\Psi_{i,j,n}=\Phi_{j,n}+\ti c_{i,n}$. Thus, we can prove the following fact.

\begin{claim}\label{claim2} $\Psi_{1,j,n} \to a_jY_{0j}$ for $j=1,\dots,m_1$  and $\Psi_{2,j,n} \to a_jY_{0j}$ for $j=m_1+1,\dots,m$, weakly in $H_{\al_j}(\R^2)$ and strongly in $L_{\al_j}(\R^2)$ as $n\to+\infty$ for some constant $a_{j}\in\R$, $j=1,\dots,m$.
\end{claim}

\begin{proof}[\dem]
From the previous computations, it is clear that in $\Om_{j,n}$
$$\lap \Psi_{1,j,n}+\de_j^2K_1(\xi_j+\de_j y ) \Psi_{1,j,n}+\de_j^2K_2(\xi_j+\de_j y ) \lf(\Psi_{1,j,n}-\ti c_{1,n}+\ti c_{2,n}\rg)=\de_j^2h_n(\xi_j+\de_j y)$$
and
$$\lap \Psi_{2,j,n}+\de_j^2K_1(\xi_j+\de_j y ) \lf(\Psi_{2,n}-\ti c_{2,n}+\ti c_{1,n}\rg) + \de_j^2K_2(\xi_j+\de_j y )  \Psi_{1,j,n}=\de_j^2h_n(\xi_j+\de_j y).$$
Furthermore, $\{\ti c_{i,n}\}$ is a bounded sequence in $\R$, so it follows that $\{\Psi_{i,j,n}\}_n$ is bounded in $H_{\al_j}(\R^2)$ for $i=1,2$ and $j=1,\dots,m$. Also, we have that
$$\int_{\Om_{j,n}} (\de_j^2|h_n(\xi_j+\de_j y)|)^p\, dy=\de_j^{2p-2}\int_{\Om_{\e_n} }|h_n(x)|^p\, dx=\de_i^{2p-2}\|h_n\|_p^p=o(1).$$
Therefore, taking into account \eqref{dik12} we deduce that $\Psi_{i,j,n}\to\Psi_j^*$ as $n\to+\infty$ with $i=1$ if $j=1,\dots,m_1$ and $i=2$ if $j=m_1+1,\dots,m$, where $\Psi^*_j$ is a solution to 
$$\lap\Psi+{2\al_j^2|y|^{\al_j-2}\over(1+|y|^{\al_j})^2}\Psi=0,\qquad j=1,\dots,m,\qquad  \text{in $\R^2\sm\{0\}$}.$$
It is standard that $\Psi_j^*$, $j=1,\dots,m$, extends to a solution in the whole $\R^2$. Hence, by using symmetry assumptions if necessary, we get that $\Psi^*_j=a_jY_{0j}$ for some constant $a_{j}\in\R$, $j=1,\dots,m$.
\end{proof}

For the next step we construct some suitable test functions. To this aim, introduce the coefficients $\gamma_{ij}$'s and $\ti\gamma_{ij}$'s, $i,j=1,\dots,m$, as the solution of the linear systems
\begin{equation}\label{gamaij}
\gamma_{ij}\lf[-{1\over 2\pi}\log\e_i+H(\xi_i,\xi_i)\rg]+\sum_{k=1,k\ne i}^m \gamma_{kj} G(\xi_k,\xi_i)=
\begin{cases}
2& \text{if } i=j\\
0& \text{if } i\ne j
\end{cases}
\end{equation}
and
\begin{equation}\label{gamatij}
\ti\gamma_{ij}\lf[-{1\over 2\pi}\log\e_i+H(\xi_i,\xi_i)\rg]+\sum_{k=1,k\ne i}^m \ti \gamma_{kj} G(\xi_k,\xi_i)=\begin{cases}
\ds {4\over 3} \al_j\log\de_j +{8\over 3} + {8\pi\over 3}\al_j H(\xi_j,\xi_j)& \text{if } i=j\\[0.3cm]
\ds {8\pi\over 3}\al_j G(\xi_i,\xi_j),& \text{if } i\ne j,
\end{cases}
\end{equation}
respectively. Notice that both systems \eqref{gamaij} and \eqref{gamatij} are diagonally dominant, system \eqref{gamaij} has solutions
$$\gamma_{ij}=\begin{cases}
\ds -{4\pi\over \log\e_j}+O\Big({1\over |\log\e|^2}\Big)=-{2\pi(\al_j-2)\over \log \e} +O\Big({1\over |\log\e|^2}\Big)&\text{for } i=j\\[0.3cm]
\ds O\Big({1\over |\log\e|^2}\Big)&\text{for }\ i\ne j
\end{cases}$$
and for the system \eqref{gamatij} we get
$$\ti\gamma_{ij}=\begin{cases}
\ds -{8\pi\over 3}{\al_j\log\de_j\over\log\e_j}+O\Big({1\over |\log\e|}\Big)=-{4\pi\over 3}(\al_j-2)+O\Big({1\over |\log\e|}\Big)&\text{for } i=j\\[0.3cm]
\ds O\Big({1\over |\log\e|}\Big)&\text{for }\ i\ne j
\end{cases}.$$
Here, we have used \eqref{choice1}. Consider now for any $j\in\{1,\dots, m\}$ the functions $\ds \eta_{0j}(x)= - \dfrac{2\de_j^{\al_j} }{\de_j^{\al_j}+|x-\xi_j|^{\al_j}}$
and
$$\eta_j(x)={4\over 3}\log(\de_j^{\al_j}+|x-\xi_j|^{\al_j}){\de_j^{\al_j}-|x-\xi_j|^{\al_j}\over \de_j^{\al_j}+|x-\xi_j|^{\al_j} }+ {8\over 3}{\de_j^{\al_j} \over \de_j^{\al_j}+|x-\xi_j|^{\al_j} },$$
 so that 
 $$\lap \eta_{0j} + |x-\xi_j|^{\al_j-2}e^{U_j} \eta_{0j}=-|x-\xi_j|^{\al_j}e^{U_j}\quad\text{ and }\quad\lap\eta_j+|x - \xi_j|^{\al_j-2}e^{U_j}\eta_j=|x-\xi_j|^{\al_j-2}e^{U_j}Z_{0j},$$
where $Z_{0j}(x)=Y_{0j}(\de_j^{-1} [x-\xi_j])=\dfrac{\de_j^{\al_j}-|x-\xi_j|^{\al_j}}{\de_j^{\al_j}+|x-\xi_j|^{\al_j}}$. Notice that $\eta_{0j} + 1= -Z_{0j}$ and, by similar arguments as to obtain expansion \eqref{pui}, we have that the following fact.

\begin{lemma}
There hold
$$P_\e \eta_{0j}=\eta_{0j} + \sum_{i=1}^m \gamma_{ij} G(x,\xi_i)+O(\e^{\ti\sigma})\quad\text{ and }\quad P_\e \eta_{j}=\eta_{j}+{8\pi\over 3}\al_j H(x,\xi_j)-\sum_{i=1}^m \ti\gamma_{ij} G(x,\xi_i)+O(\e^{\ti\sigma})$$ 
uniformly in $\Om_\e$ for some $\ti\sigma>0$.
\end{lemma}

\begin{proof}[\dem] On one hand, the harmonic function $\ds f(x)=P_\e \eta_{0j}(x)-\eta_{0j}(x)-\sum_{i=1}^m\gamma_{ij}G(x,\xi_i)$ satisfies \linebreak $\ds f(x)={2\de_j^{\al_j}\over \de_j^{\al_j}+|x-\xi_j|^{\al_j}}=O(\de_j^{\al_j})$ on $\fr\Om$ and
\begin{equation*}
\begin{split}
f(x)&={2\de_j^{\al_j}\over \de_j^{\al_j}+\e^{\al_j}} - \gamma_{jj}\lf[- {1\over 2\pi}\log \e_j+H(\xi_j,\xi_j)+O(\e_j)\rg] - \sum_{i=1,i\ne j}^m\gamma_{ij}\lf[G(\xi_j,\xi_i)+O(\e_j)\rg]\\
&=O\lf(\Big({\e_j\over \de_j}\Big)^{\al_j}\rg)+\gamma_{jj}O(\e_j)+\sum_{i=1,i\ne j}^m\gamma_{ij}O(\e_j)
\end{split}
\end{equation*}
on $ \fr B(\xi_j,\e_j)$ by using the first equation in \eqref{gamaij} and  
\begin{equation*}
\begin{split}
f(x)&={2\de_j^{\al_j}\over \de_j^{\al_j}+|x-\xi_i|^{\al_j}} - \gamma_{ij}\lf[- {1\over 2\pi}\log \e_i+H(\xi_i,\xi_i)+O(\e_i)\rg] - \sum_{k=1,k\ne i}^m\gamma_{kj}\lf[G(\xi_i,\xi_k)+O(\e_i)\rg]\\
&=O\lf(\de_j^{\al_j}\rg)+\gamma_{ij}O(\e_i)+\sum_{k=1,k\ne i}^m\gamma_{kj}O(\e_i)
\end{split}
\end{equation*}
on $\fr B(\xi_i,\e_i)$ for $i\ne j$ by using the second equation in \eqref{gamaij}. Therefore, by the maximum principle we deduce the expansion of $P_\e \eta_{0j}$.

On the other hand, similarly as above the harmonic function 
$$\ds \ti f(x)=P_\e \eta_{j}(x)-\eta_{j}(x)-{8\pi\over 3}\al_jH(x,\xi_j)+\sum_{i=1}^m\ti\gamma_{ij}G(x,\xi_i)$$
satisfies 
$$ \ti f(x) =-{4\over 3}\log(\de_j^{\al_j}+|x-\xi_j|^{\al_j}) \bigg[-1+  {2\de_j^{\al_j} \over \de_j^{\al_j}+|x-\xi_j|^{\al_j} }  \bigg]   - {4\over 3}\al_j\log|x-\xi_j| + O(\de_j^{\al_j}) =O(\de_j^{\al_j})\ \text{ on $\fr\Om$}$$
and
\begin{equation*}
\begin{split}
\ti f(x)
&=-{4\over 3}\bigg[\log \de_j^{\al_j}+  \log\(1+\Big({\e_j\over \de_j}\Big)^{\al_j}\) \bigg] \cdot  {1 -\Big(\dfrac{\e_j}{ \de_j}\Big)^{\al_j}\over 1+\Big(\dfrac{\e_j}{ \de_j}\Big)^{\al_j} }  - {8\over 3}\cdot  {1\over 1+\Big(\dfrac{\e_j}{ \de_j}\Big)^{\al_j} }  -{8\pi\over 3}\al_j H(\xi_j,\xi_j)\\
&\ \quad+ \ti \gamma_{jj}\lf[- {1\over 2\pi}\log \e_j+H(\xi_j,\xi_j)\rg]+\sum_{i=1,i\ne j}^m\ti\gamma_{ij}G(\xi_j,\xi_i) 
+ \ti \gamma_{jj} O(\e_j) +\sum_{i=1,i\ne j}^m\ti\gamma_{ij} O(\e_j) \\
&=O\lf(\Big({\e_j\over \de_j}\Big)^{\al_j}\rg) + O\lf(\Big({\e_j\over \de_j}\Big)^{\al_j}|\log\de_j| \rg)+ O(\e_j)+\ti\gamma_{jj}O(\e_j)+\sum_{i=1,i\ne j}^m\ti \gamma_{ij}O(\e_j)
\end{split}
\end{equation*}
on $ \fr B(\xi_j,\e_j)$, by using the first equation \eqref{gamatij} and
\begin{equation*}
\begin{split}
\ti f(x) 
&=\bigg[-{4\over 3}\al_j\log|x-\xi_j| -  {4\over 3} \log\(1+{\de_j^{\al_j}\over |x-\xi_j|^{\al_j} }\)  \bigg] \bigg[-1+  {2\de_j^{\al_j} \over \de_j^{\al_j}+|x-\xi_j|^{\al_j} }  \bigg]   - {8\over 3} {\de_j^{\al_j} \over \de_j^{\al_j}+|x-\xi_j|^{\al_j} }	\\
&\quad - {8\pi\over 3}\al_jH(x,\xi_j)+  \ti \gamma_{ij}\lf[- {1\over 2\pi}\log \e_i+H(\xi_i,\xi_i)\rg] +\sum_{k=1,k\ne i}^m\ti\gamma_{kj} G(\xi_i,\xi_k)  +\ti \gamma_{ij} O(\e_i)  +\sum_{k=1,k\ne i}^m\ti\gamma_{kj} O(\e_i) \\
&=O\lf(\de_j^{\al_j}\rg)+\gamma_{ij}O(\e_i)+\sum_{k=1,k\ne i}^m\gamma_{kj}O(\e_i)
\end{split}
\end{equation*}
on  $ \fr B(\xi_i,\e_i)$ for $i\ne j$ by using the second equation \eqref{gamatij}. Therefore, by the maximum principle we deduce the expansion of $P_\e \eta_{j}$.
\end{proof}

Denote $\ti c_i=\displaystyle \lim_{n\to+\infty} \ti c_{i,n}$ for $i=1,2$, up to a subsequence if necessary. Hence, we get that
\begin{equation}\label{cpsi}
\Phi_{j,n} \to a_jY_{0j}-\ti c_1, \ \text{ for } j=1,\dots, m_1,\quad\text{and}\quad \Phi_{j,n} \to a_jY_{0j}-\ti c_2, \  \text{ for } j=m_1+1,\dots,m,
\end{equation}
weakly in $H_{\al_j}(\R^2)$ and strongly in $L_{\al_j}(\R^2)$, since $\Phi_{j,n}=\Psi_{i,j,n}-\ti c_{i,n}$.

\begin{claim}\label{claim3}
There hold that $(\al_j-1)a_{j} + 2\ti c_i=0$ either for $i=1$ and all $j=1,\dots,m_1$ or for $i=2$ and all $j=m_1+1,\dots, m$.
\end{claim}

\begin{proof}[\dem] To this aim define the following test function $P_\e Z_j$, where $Z_j=\eta_j+\gamma_j^*\eta_{0j}$ and $\gamma_j^*$ is given by
$$ \gamma_j^*=\dfrac{\ds \dfrac{\ti\gamma_{jj}}{2\pi}\log\de_j +\(\dfrac{8\pi}{3}\al_j-\ti\gamma_{jj}\) H(\xi_j,\xi_j) -  \sum_{i=1,i\ne j} ^m \ti\gamma_{ij} G(\xi_i,\xi_j) }{\ds 1-\gamma_{jj}H(\xi_j,\xi_j)- \sum_{i=1,i\ne j} ^m \gamma_{ij} G(\xi_i,\xi_j) + \dfrac{\gamma_{jj}}{2\pi}\log\de_j},$$
so that
\begin{equation}\label{gamajs}
\gamma_j^*= \(\dfrac{8\pi}{3}\al_j-\ti\gamma_{jj}+ \gamma_{jj}\gamma_j^*\) H(\xi_j,\xi_j) -  \sum_{i=1,i\ne j} ^m (\ti\gamma_{ij} - \gamma_j^* \gamma_{ij}) G(\xi_i,\xi_j) +\dfrac{1}{2\pi}(\ti\gamma_{jj}- \gamma_{j}^* \gamma_{jj})\log\de_j .
\end{equation}
Thus,  from the assumption on $h_n$, $|\log \epsilon_n|\ \|h_n\|_*=o(1)$, we get the above relation between $a_j$ and $\ti c_i$ either for $i=1$ and all $j=1,\dots,m_1$ or for $i=2$ and all $j=m_1+1,\dots, m$.
Furthermore, from \eqref{choice1} and the expansions for $\gamma_{jj}$ and $\ti\gamma_{jj}$ we obtain that
$$\gamma_j^*=\dfrac{\ds \dfrac{\ti\gamma_{jj}}{2\pi}\log\de_j+O(1)}{\ds1+\dfrac{ \gamma_{jj}}{2\pi}\log\de_j + O\Big({1\over |\log\e|} \Big)} = - {\al_j-2\over 3}\log \e +O(1).$$
Notice that $P_\e Z_j$ expands as
\begin{align}\label{pzj}
P_\e Z_j 
=&\,Z_{j}+{8\pi\over 3}\al_j H(x,\xi_j)- \ti\gamma_{jj}\(-{1\over 2\pi}\log|x - \xi_j| + H(x,\xi_j)\) -\sum_{i=1,i\ne j}^m \ti\gamma_{ij} G(x,\xi_i) \nonumber\\
&+O(\e^{\ti\sigma}) + \gamma_j^*\lf[\gamma_{jj}\(-{1\over 2\pi}\log|x - \xi_j| + H(x,\xi_j)\) + \sum_{i=1,i\ne j}^m \gamma_{ij} G(x,\xi_i)+O(\e^{\ti\sigma})\rg] \nonumber\\
=&\,Z_{j}+\( {8\pi\over 3} \al_j - \ti\gamma_{jj} + \gamma_{jj}\gamma_j^*\) H(x,\xi_j) -\sum_{i=1,i\ne j}^m \(\ti\gamma_{ij} -  \gamma_{ij}\gamma_j^* \)G(x,\xi_i)\\
& + {1\over 2\pi}\(\ti\gamma_{jj}-\gamma_j^*\gamma_{jj}\) \log|x - \xi_j| +O(\e^{\ti\sigma}) + \gamma_{j}^* O(\e^{\ti\sigma}).\nonumber
\end{align}

\medskip \noindent Assume that $i=1$ for all $j=1,\dots,m_1$ or $i=2$ for all  $j=m_1+ 1,\dots, m $. Multiplying equation \eqref{eqphin} by $P_\e Z_j$ and integrating by parts we obtain that
\begin{equation*}
\begin{split}
\int_{\Om_\e} hP_\e Z_{j}=&\ \int_{\Om_\e}\lap Z_{j} \lf[\psi_i - \ti c_{i,n}\rg] + \int_{\Om_\e} \lf[
K_1\psi_1+ K_2\psi_2\rg]  P_\e Z_{j}  ,
\end{split}
\end{equation*}
in view of $P_\e Z_{j}=0$ and $\psi_i=\ti c_{i,n}$ on $\fr\Om_\e$ and
\begin{equation*}
\begin{split}
\int_{\Om_\e}\lap\psi_i P_\e Z_{j}&=\int_{\Om_\e}\psi_i\lap P_\e Z_{j}-\psi_i\bigg|_{\fr\Om_\e}\int_{\Om_e}\lap PZ_{j}=\int_{\Om_\e} \lap Z_{j} \lf[\psi_i - \ti c_{i,n}\rg].
\end{split}
\end{equation*}
Furthermore, we have that
\begin{equation*}
\begin{split}
\int_{\Om_\e} hPZ_{j} 
=&\,\int_{\Om_\e} \lf[\psi_i - \ti c_{i,n}\rg] \Big[|x-\xi_j|^{\al_j-2}e^{U_j}Z_{0j} - |x-\xi_j|^{\al_j-2} e^{U_j}\eta_j\\
& +\gamma_j^*\( - |x-\xi_j|^{\al_j-2}e^{U_j}  - |x-\xi_j|^{\al_j-2} e^{U_j}\eta_{0j}\) \Big]+\int_{\Om_\e} \lf[
K_1\psi_1+ K_2\psi_2\rg]  P_\e Z_{j} \\
=&\,\int_{\Om_\e} \psi_i |x-\xi_j|^{\al_j-2}e^{U_j}Z_{0j} +\int_{\Om_\e} |x-\xi_j|^{\al_j-2} e^{U_j}\psi_i \( PZ_j-Z_j - \gamma_j^*\) \\
& +\int_{\Om_\e} \(K_1\psi_1+K_2\psi_2 - |x-\xi_j|^{\al_j-2}e^{U_j}\psi_i \)  P_\e Z_{j} -\ti c_{i,n} \int_{\Om_\e} |x-\xi_j|^{\al_j-2} e^{U_j} \( Z_{0j}-\eta_j + \gamma_j^* Z_{0j}\),  
\end{split}
\end{equation*}
in view of 
$$\lap\eta_j + \gamma_j^*\lap\eta_{0j}= |x-\xi_j|^{\al_j-2}e^{U_j}\lf[Z_{0j} - \eta_j  +\gamma_j^*\( - 1  - \eta_{0j}\)\rg] =|x-\xi_j|^{\al_j-2}e^{U_j}\lf[Z_{0j} - \eta_j  +\gamma_j^*Z_{0j}\rg] .$$
Now, estimating every integral term we find that $\ds \int_{\Om_\e} hPZ_{j}=O\lf(\, |\log\e|\,\|h\|_p\rg)=o\lf( 1 \rg)$ for all $j=1,\dots, m$, in view of $PZ_{j}=O(|\log\e|)$ and $G(x,\xi_k)=O(|\log\e_k|)$. Next, by scaling we obtain that either for $i=1$ and all $j=1,\dots,m_1$ or $i=2$ and all $j=m_1+ 1,\dots, m $ it holds
$$\int_{\Om_\e} \psi_i |x-\xi_j|^{\al_j-2}e^{U_j}Z_{0j}=\int_{\Om_{j,n } }  {2\al_j^2|y|^{\al_j-2}\over (1+|y|^{\al_j})^2 }\Psi_{i,j,n}Y_{0j}\, dy =a_j\int_{\R^2 }  {2\al_j^2|y|^{\al_j-2}\over (1+|y|^{\al_j})^2 }Y_{0j}^2 \, dy +o(1).$$
Note that
$$\int_{\R^2} {2\al_j^2|y|^{\al_j-2}\over (1+|y|^{\al_j})^2}Y_{0j}^2=\int_{\R^2} {2\al_j^2|y|^{\al_j-2}\over (1+|y|^{\al_j})^2}\({1-|y|^{\al_j} \over 1+|y|^{\al_j} }\)^2dy={4\pi\over 3}\al_j$$
and 
$$\int_{\R^2} {2\al_j^2|y|^{\al_j-2}\over (1+|y|^{\al_j})^2}Y_{0j}\log|y|=\int_{\R^2} {2\al_j^2|y|^{\al_j-2}\over (1+|y|^{\al_j})^2}\ {1-|y|^{\al_j} \over 1+|y|^{\al_j} }\ \log|y|\, dy= -4\pi .$$
Also, by using \eqref{gamajs}-\eqref{pzj} we get that
\begin{equation*}
\begin{split}
 \int_{\Om_\e} |x-\xi_j|^{\al_j-2} e^{U_j} \psi_i&\, \( PZ_j-Z_j - \gamma_j^*\)= \int_{\Om_\e} |x-\xi_j|^{\al_j-2} e^{U_j} \psi_i\bigg[ PZ_j-Z_j -\( {8\pi\over 3} \al_j - \ti\gamma_{jj} + \gamma_{jj}\gamma_j^*\) H(x,\xi_j)\\
 &\,+ \sum_{i=1,i\ne j}^m \(\ti\gamma_{ij} -  \gamma_{ij}\gamma_j^* \)G(x,\xi_i)- {1\over 2\pi}\(\ti\gamma_{jj}-\gamma_j^*\gamma_{jj}\) \log|x - \xi_i| \bigg]\\
&\,+\int_{\Om_\e}\psi_i | x-\xi_j|^{\al_j-2}e^{U_j}  \( {8\pi\over 3} \al_j - \ti\gamma_{jj} + \gamma_{jj}\gamma_j^*\)  \lf[H(x,\xi_j) - H(\xi_j,\xi_j)\rg]  \\
&\,+\int_{\Om_\e}\psi_i | x-\xi_j|^{\al_j-2}e^{U_j} \sum_{i=1,i\ne j}^m  \(\ti\gamma_{ij} -  \gamma_{ij}\gamma_j^* \)  \lf[G(\xi_i,\xi_j) - G( x ,\xi_j)\rg] \\ 
&\,+ {1\over 2\pi}\(\ti\gamma_{jj}-\gamma_j^*\gamma_{jj}\)\int_{\Om_\e}\psi_i | x-\xi_j|^{\al_j-2}e^{U_j}   \lf[\log|x - \xi_i| - \log\de_j\rg] \\
=&\,  \int_{\Om_\e} |x-\xi_j|^{\al_j-2} e^{U_j} \psi_i \, O(\e^{\ti\sigma}) +\int_{\Om_{j,n } }  {2\al_j^2|y|^{\al_j-2}\over (1+|y|^{\al_j})^2 }\Psi_{i,j,n} O(\de_j|y|)\, dy \\
&\, + {1\over 2\pi} \(\ti\gamma_{jj}-\gamma_j^*\gamma_{jj}\)\int_{\Om_{j,n } }  {2\al_j^2|y|^{\al_j-2}\over (1+|y|^{\al_j})^2 }\Psi_{i,j,n} \log |y| \, dy\\
=&\,-{\al_j(\al_j-2)\over 3} \	 a_j\int_{\R^2 }  {2\al_j^2|y|^{\al_j-2}\over (1+|y|^{\al_j})^2 }Y_{0j} \log |y| \, dy + o(1),
\end{split}
\end{equation*}
in view of
\begin{equation*}
{1\over 2\pi} \(\ti\gamma_{jj}-\gamma_j^*\gamma_{jj}\)=-{\al_j(\al_j-2)\over 3} + O\Big({1\over |\log\e|} \Big).
\end{equation*}
Furthermore, using \eqref{K12} we have that
\begin{equation*}
\begin{split}
\int_{\Om_\e} \(K_i - |x-\xi_j|^{\al_j-2}e^{U_j}\)  P_\e Z_{j} \psi_i=& P_\e Z_j(\xi_j+\de_j y)\, dy\\
 \sum_{l\ne j} \int_{ \Om_{l,n} } {2\al_l^2|y|^{\al_l-2}\over (1+|y|^{\al_l})^2}  \Psi_{i,l,n} P_\e Z_j(\xi_l+\de_l y)\, dy=o(1)
\end{split}
\end{equation*}
since 
for $l\ne j$ and $y\in \de_l^{-1}(B_l-\xi_l)$ it holds
\begin{equation*}
\begin{split}
P_\e Z_j(\xi_l+\de_l y)=&\, Z_{j}(\xi_l+\de_l y) +{8\pi\over 3} \al_j  H(\xi_l+\de_l y,\xi_j)    -\sum_{k=1,k\ne j}^m \(\ti\gamma_{kj} -  \gamma_{kj}\gamma_j^* \)G(\xi_l+\de_l y,\xi_k)\\
& +  {1\over 2\pi}\(\ti\gamma_{jj}-\gamma_j^*\gamma_{jj}\) \log|\xi_l+\de_l y - \xi_j| -\( \ti\gamma_{jj} - \gamma_{jj}\gamma_j^*\) H(\xi_l+\de_l y,\xi_j) +O(\e^{\ti\sigma})\\
=&\, {8\pi\over 3}\al_j G(\xi_l,\xi_j ) +O(\de_j^{\al_j} + \de_l|y|) - \(\ti\gamma_{lj} -  \gamma_{lj}\gamma_j^* \)\(-{1\over 2\pi} \log|\de_ly| + H(\xi_l+\de_l y,\xi_l)\) \\
&-\sum_{k=1,k\ne l}^m \(\ti\gamma_{kj} -  \gamma_{kj}\gamma_j^* \)\( G(\xi_l,\xi_k) + O(\de_l |y|)\) \\ 
=&\, {1\over 2\pi}  \(\ti\gamma_{lj} -  \gamma_{lj}\gamma_j^* \)  \log|\de_ly|  + {8\pi\over 3}\al_j G(\xi_i,\xi_j ) - \(\ti\gamma_{lj} -  \gamma_{lj}\gamma_j^* \) H(\xi_l,\xi_l)  \\
&- \sum_{k=1,k\ne l}^m \(\ti\gamma_{kj} -  \gamma_{kj}\gamma_j^* \)G(\xi_l,\xi_k) + O(\de_l |y|),
\end{split}
\end{equation*}
and using that $\ti\gamma_{lj} -  \gamma_{lj}\gamma_j^*=O(|\log \e|^{-1})$ for $l\ne j$, we deduce that
\begin{equation*}
\begin{split}
\int_{ \Om_{l,n} } &{2\al_l^2|y|^{\al_l-2}\over (1+|y|^{\al_l})^2} \Psi_{i,l,n} P_\e Z_j(\xi_l+\de_l y)\, dy\\
=&\,{1\over 2\pi} \(\ti\gamma_{lj} -  \gamma_{lj}\gamma_j^* \)\log\de_l   \int_{\Om_{l,n} } {2\al_l^2|y|^{\al_l-2}\over (1+|y|^{\al_l})^2} \Psi_{i,l,n}  \, dy + {1\over 2\pi} \(\ti\gamma_{lj} -  \gamma_{lj}\gamma_j^* \)  \int_{ \Om_{l,n} } {2\al_l^2|y|^{\al_l-2}\over (1+|y|^{\al_l})^2}\Psi_{i,l,n}   \log| y| \, dy  \\
&\, - \(\text{bounded constant} \)  \int_{B_l-\xi_l\over\de_l } {2\al_l^2|y|^{\al_l-2}\over (1+|y|^{\al_l})^2}  \Psi_{i,l,n}\, dy + O\( \de_l \int_{ \Om_{l,n} } {2\al_l^2|y|^{\al_l-1}\over (1+|y|^{\al_l})^2} \, dy\)\\
 = &\, o(1).
\end{split}
\end{equation*}
Notice that
$$\int_{\Om_{j,n} } {2\al_j^2|y|^{\al_j-2}\over (1+|y|^{\al_j})^2 }  \Psi_{i,j,n}(y)\, dy=a_j\int_{\R^2 } {2\al_j^2|y|^{\al_j-2}\over (1+|y|^{\al_j})^2 }  Y_{0j}(y)\, dy+o(1)=o(1),$$
since
$$\int_{\R^2 } {2\al_j^2|y|^{\al_j-2}\over (1+|y|^{\al_j})^2 }  Y_{0j}(y)\, dy=\int_{\R^2 } {2\al_j^2|y|^{\al_j-2}\over (1+|y|^{\al_j})^2 } \cdot {1-|y|^{\al_j} \over 1+|y|^{\al_j} }\, dy=0.$$
If either $i=2$ and $j=1,\dots,m_1$ or $i=1$ and $j=m_1+1,\dots,m$, from similar computations as above we get that
\begin{equation*}
\begin{split}
\int_{\Om_\e} K_i \psi_i P_\e Z_{j}&= \sum_{l} \int_{ \Om_{l,n} } {2\al_l^2|y|^{\al_l-2}\over (1+|y|^{\al_l})^2} \Psi_{i,l,n} P_\e Z_j(\xi_l+\de_l y)\, dy  = o(1).
\end{split}
\end{equation*}
Here, we sum over $l=1,\dots, m_1$ for $i=1$ and $l=m_1+1,\dots, m$ for $i=2$. 
Besides, similarly as above we obtain that
\begin{align*}
\int_{\Om_\e} |x-\xi_j|^{\al_j-2} e^{U_j}&\( Z_{0j}-\eta_j + \gamma_j^* Z_{0j}\) 
= (1+\gamma_j^*)\int_{\Om_\e} |x-\xi_j|^{\al_j-2} e^{U_j}  Z_{0j} - \int_{\Om_\e} |x-\xi_j|^{\al_j-2} e^{U_j} \eta_j\\
&= (1+\gamma_j^*)\lf[ \int_{B_{r\over \de_j}(0)\sm B_{\e_j\over\de_j}(0)} {2\al_j^2|y|^{\al_j-2}\over(1+|y|^{\al_j})^2}  {1-|y|^{\al_j}\over 1+|y|^{\al_j}}   \ dy + O(\de_j^{\al_j}) \rg]\\
&\qquad - \int_{\Om_{j,n } }  {2\al_j^2|y|^{\al_j-2}\over (1+|y|^{\al_j})^2 } \lf[ {4\over 3}  \log\(\de_j^{\al_j} + \de_j^{\al_j}|y|^{\al_j}\)  Y_{0j}(y) + {8\over 3}{1\over 1+|y|^{\al_j} }\rg]dy\\
&= O(\e^{\ti\sigma} |\log\e|)- {4\over 3} \al_j \log \de_j  \int_{\Om_{j,n } }  {2\al_j^2|y|^{\al_j-2}\over (1+|y|^{\al_j})^2 } Y_{0j}(y) \, dy \\
&\qquad- {4\over 3}  \int_{\Om_{j,n } }  {2\al_j^2|y|^{\al_j-2}\over (1+|y|^{\al_j})^2 }  Y_{0j}(y) \log\(1 +|y|^{\al_j}\) \, dy -  {8\over 3} \int_{\Om_{j,n } }  {2\al_j^2|y|^{\al_j-2}\over (1+|y|^{\al_j})^2 } {1\over 1+|y|^{\al_j} }\, dy \\
&= - {8\pi\over 3}\al_j +o(1),
\end{align*}
in view of
$$\int_{\R^2}  {2\al_j^2|y|^{\al_j-2}\over (1+|y|^{\al_j})^2 }  Y_{0j}(y) \log\(1 +|y|^{\al_j}\) \, dy  =-2\pi\al_j\qquad \text{and}\qquad\int_{\R^2 }  {2\al_j^2|y|^{\al_j-2}\over (1+|y|^{\al_j})^2 } {1\over 1+|y|^{\al_j} }\, dy =2\pi\al_j.$$
Therefore, we conclude that
$$o(1)= a_j\lf({4\pi\over 3}\al_j+o(1)\rg) - {\al_j(\al_j-2)\over 3}a_j \lf( - 4\pi +o(1)\rg) - \ti c_{i,n} \lf( - {8\pi\over 3}\al_j+o(1)\rg) +o(1),$$
and hence $(\al_j-1)a_j + 2\ti c_{i}=0$ either for $i=1$ and all $j=1,\dots,m_1$ or $i=2$ and all $j=m_1+ 1,\dots, m $.
\end{proof}

\begin{claim}\label{claim4}
There hold that 
\begin{equation}\label{saj}
\sum_{j=1}^{m_1} \al_j(\al_j-2) a_{j}=0\qquad \text{and}\qquad \sum_{j=m_1+1}^m\al_j(\al_j-2)a_j=0.
\end{equation} 
Hence, from Claim \ref{claim3} it follows that $\ti c_i=0$ for $i=1,2$ and then $a_j=0$ for all $j=1,\dots,m$.
\end{claim}

\begin{proof}[\dem] Similarly as above, let us use suitable test functions to get the claimed relations. Consider the functions $Z_{0j}(x)=Y_{0j}(\de_j^{-1} [x-\xi_j])$ so that $-\lap Z_{0j}=|x-\xi_j|^{\al_j-2}e^{U_j} Z_{0j}$, for all $j=1,\dots, m$. From the fact that $Z_{0j} =-\eta_{0j} - 1$, we have that
$$P_\e Z_{0j}=Z_{0j}+1-\sum_{i=1}^m \gamma_{ij} G(x,\xi_i)+O(\e^{\ti\sigma})$$ 
for some $\ti\sigma>0$, where the $\gamma_{ij}$'s, $i,j=1,\dots,m$, satisfy the diagonal dominant system \eqref{gamaij}. Assume that either $i=1$ for all $j=1,\dots,m_1$ or $i=2$ for all $j=m_1+1,\dots, m $. Similarly as above, multiplying equation \eqref{eqpsi} by $\gamma_{jj}^{-1}P_\e Z_{0j}$ and integrating by parts we obtain that
\begin{equation}\label{tpz0j}
\begin{split}
\gamma_{jj}^{-1}\int_{\Om_\e} hPZ_{0j}=&\ \int_{\Om_\e}\lf([K_1\psi_1+ K_2 \psi_2]\gamma_{jj}^{-1}P_\e Z_{0j}-|x-\xi_j|^{\al_j-2}e^{U_j}\gamma_{jj}^{-1}Z_{0j}\psi_i\rg)\\
&\ +\psi_i\bigg|_{\fr \Om_\e}\int_{\Om_\e}\gamma_{jj}^{-1}|x-\xi_j|^{\al_j-2}e^{U_j} Z_{0j}.
\end{split}
\end{equation}
Now, estimating every integral term we find that $\ds \gamma_{jj}^{-1}\int_{\Om_\e} hPZ_{0j}=O\lf(|\log\e|\, \|h\|_p\rg)=o(1),$ in view of $PZ_{0j}=O(1)$, $G(x,\xi_k)=O(|\log\e_k|)$ and the choice of $\gamma_{jj}$. Next, we obtain that 
$$\int_{\Om_\e}\gamma_{jj}^{-1}|x-\xi_j|^{\al_j-2}e^{U_j} Z_{0j}=\gamma_{jj}^{-1}\int_{B_{r\over \de_j}(0)\sm B_{\e_j\over\de_j}(0)} {2\al_j^2|y|^{\al_j-2}\over(1+|y|^{\al_j})^2} {1-|y|^{\al_j}\over 1+|y|^{\al_j}} \ dy + O(\de_j^{\al_j}|\log\e|)=o(1).$$
Also, we have that
\begin{align*}
\gamma_{jj}^{-1} \int_{\Om_\e}\lf( K_i  P_\e Z_{0j}-|x-\xi_j|^{\al_j-2}e^{U_j} Z_{0j}\rg) \psi_i  = &\, \gamma_{jj}^{-1} \int_{\Om_\e}|x-\xi_j|^{\al_j-2}e^{U_j} \(P_\e Z_{0j} - Z_{0j}\)\psi_i \\
&+ \gamma_{jj}^{-1} \int_{\Om_\e}\lf( K_i  -|x-\xi_j|^{\al_j-2}e^{U_j}\rg) P_\e Z_{0j}  \psi_i.
\end{align*}
We estimate the first term as
\begin{align*}
\gamma_{jj}^{-1} \int_{\Om_\e} &\, |x-\xi_j|^{\al_j-2}e^{U_j}  \(P_\e Z_{0j} - Z_{0j}\)\psi_i = \gamma_{jj}^{-1} \int_{\Om_\e}|x-\xi_j|^{\al_j-2}e^{U_j} \(1-\sum_{i=1}^m \gamma_{ij} G(x,\xi_i)+O(\e^{\ti\sigma})\)\psi_i \\
=&\,  \gamma_{jj}^{-1}  \int_{\Om_\e} |x-\xi_j|^{\al_j-2}e^{U_j} \psi_i  +\int_{\Om_\e} |x-\xi_j|^{\al_j-2}e^{U_j} \({1\over 2\pi}\log|x-\xi_j| - H(x,\xi_j)\) \psi_i\\
&\, -\int_{\Om_\e} |x-\xi_j|^{\al_j-2}e^{U_j} \sum_{i=1,i\ne j}^m \gamma_{jj}^{-1}\gamma_{ij}  G(x,\xi_i) \psi_i + O(\e^{\ti \sigma} |\log \e|) \\
=&\, \gamma_{jj}^{-1}\int_{\Om_{j,n} } {2\al_j^2|y|^{\al_j-2}\over (1+|y|^{\al_j})^2 } \Psi_{i,j,n}(y)\, dy + {1\over 2\pi } \int_{\Om_{j,n} } {2\al_j^2|y|^{\al_j-2}\over (1+|y|^{\al_j})^2 } \log|\de_j y| \Psi_{i,j,n}(y)\, dy\\
&\, -\int_{\Om_{j,n} } {2\al_j^2|y|^{\al_j-2}\over (1+|y|^{\al_j})^2 } H(\xi_j +\de_j y,\xi_j)  \Psi_{i,j,n}(y)\, dy \\
&\, - \sum_{i=1,i\ne j}^{m} \gamma_{jj}^{-1}\gamma_{ij} \int_{\Om_{j,n} } {2\al_j^2|y|^{\al_j-2}\over (1+|y|^{\al_j})^2 }  G(\xi_j+\de_j y,\xi_i) \Psi_{i,j,n}(y)\, dy + O(\e^{\ti \sigma }|\log \e|) \\
=&\, \(\gamma_{jj}^{-1} +{1\over 2\pi } \log\de_j\) \int_{\Om_{j,n} } {2\al_j^2|y|^{\al_j-2}\over (1+|y|^{\al_j})^2 } \Psi_{i,j,n}(y)\, dy + {1\over 2\pi } \int_{\Om_{j,n} } {2\al_j^2|y|^{\al_j-2}\over (1+|y|^{\al_j})^2 } \log| y| \Psi_{i,j,n}(y)\, dy\\
&\, - H(\xi_j,\xi_j) \int_{\Om_{j,n} } {2\al_j^2|y|^{\al_j-2}\over (1+|y|^{\al_j})^2 }  \Psi_{i,j,n}(y)\, dy  + O\lf (\de_j\int_{\Om_{j,n} } {2\al_j^2|y|^{\al_j-1}\over (1+|y|^{\al_j})^2 }  |\Psi_{i,j,n}(y) |\, dy +{1\over |\log\e| } \rg)  \\
=&\, \(\gamma_{jj}^{-1} + {1\over 2\pi }\log\de_j\) \int_{\Om_{j,n} } {2\al_j^2|y|^{\al_j-2}\over (1+|y|^{\al_j})^2 } \Psi_{i,j,n}(y)\, dy + {1\over 2\pi } \int_{\Om_{j,n} } {2\al_j^2|y|^{\al_j-2}\over (1+|y|^{\al_j})^2 } \log| y| \Psi_{i,j,n}(y)\, dy\\
&\, +o(1).
\end{align*}
For the next one, for $i=1$, $j=1,\dots, m_1$ we find that
\begin{align*}
\gamma_{jj}^{-1} \int_{\Om_\e}&\, \lf( K_1  -|x-\xi_j|^{\al_j-2}e^{U_j}\rg) P_\e Z_{0j}  \psi_1
  \gamma_{jj}^{-1}\sum_{l=1\atop l\ne j }^{m_1} \int_{ \Om_{l,n} } {2\al_l^2|y|^{\al_l-2} \over (1+|y|^{\al_l} )^2 } \Psi_{1,l,n}(y) P_\e Z_{0j}(\xi_l+\de_l y) \, dy=o(1)
\end{align*}
in view of 
\begin{align*}
P_\e Z_{0j}(\xi_l+\de_l y)&={2 \de_j^{\al_j} \over \de_j^{\al_j} +|\xi_l+\de_l y - \xi_j|^{\al_j} }-\sum_{k=1}^m\gamma_{kj} G( \xi_l+\de_l y,\xi_k) +O(\e^{\ti\sigma})\\
&=O\(\de_j^{\al_j}\) + \gamma_{lj}\lf({1\over 2\pi} \log |\de_l y| - H(\xi_l + \de_l |y|,\xi_l) \rg) - \sum_{k=1\atop k\ne l}^m\gamma_{kj} \(G( \xi_l,\xi_k)+ O(\de_l| y| )\) +O(\e^{\ti\sigma})
\end{align*}
for $l\ne j$ and
\begin{align*}
\gamma_{jj}^{-1}& \sum_{l=1\atop l\ne j }^{m_1} \int_{ \Om_{l,n} } {2\al_l^2|y|^{\al_l-2} \over (1+|y|^{\al_l} )^2 } \Psi_{1,l,n}(y) P_\e Z_{0j}(\xi_l+\de_l y) \, dy\\
=&\, \sum_{l=1\atop l\ne j}^{m_1} \int_{ \Om_{l,n} } {2\al_l^2|y|^{\al_l-2}\over (1+|y|^{\al_l})^2}  \Psi_{1,l,n}(y) \gamma_{jj}^{-1} \gamma_{lj} \lf( {1\over 2\pi} \log |\de_l y| - H(\xi_l,\xi_l) +O(\de_l |y|)  \rg) dy\\
&\quad-\sum_{l=1,l\ne j}^{m_1} \int_{\Om_{l,n} } {2\al_l^2|y|^{\al_l-2}\over (1+|y|^{\al_l})^2} \Psi_{1,l,n}(y) \(G(\xi_l,\xi_j)+O(\de_l |y|) \)\, dy\\
&\quad-\sum_{l=1,l\ne j}^{m_1} \int_{\Om_{l,n} } {2\al_l^2|y|^{\al_l-2}\over (1+|y|^{\al_l})^2} \Psi_{1,l,n}(y)\sum_{k=1,k\ne j,l}^m \gamma_{jj}^{-1} \gamma_{kj}  \(G(\xi_l,\xi_k)+O(\de_l |y|) \)\, dy+O(\de_j^{\al_j}|\gamma_{jj}|^{-1}) \\
=&\, o(1) .
\end{align*}
Similarly, for $i=2$, $j=m_1+1, \dots, m$ we find that
$$\gamma_{jj}^{-1} \int_{\Om_\e}\, \lf( K_2  -|x-\xi_j|^{\al_j-2}e^{U_j}\rg) P_\e Z_{0j}\psi_2=o(1). $$

On the other hand, if either $i=2$ and $j\in\{1,\dots,m_1\}$ or $i=1$ and $j\in\{m_1+1,\dots,m\}$, from similar computations as above and the expansion of $P_\e Z_{0j}(\xi_k+\de_k y)$ for $j\ne k$, we obtain that
\begin{equation*}
\begin{split}
\gamma_{jj}^{-1}  \int_{\Om_\e} & K_2\psi_2 P_\e Z_{0j} 
= \gamma_{jj}^{-1} \sum_{k=m_1+1}^m \int_{ \Om_{k,n} } {2\al_k^2|y|^{\al_k-2}\over (1+|y|^{\al_k})^2}  \Psi_{2,k,n}(y) P_\e Z_{0j}(\xi_k+\de_k y)\, dy\\
=&\, \sum_{k=m_1+1}^m \int_{ \Om_{k,n} } {2\al_k^2|y|^{\al_k-2}\over (1+|y|^{\al_k})^2}  \Psi_{2,k,n}(y) \gamma_{jj}^{-1} \gamma_{kj}\lf({1\over 2\pi}\log|\de_k y|-H(\xi_k+\de_k y,\xi_k)\rg)\\
&\quad-\sum_{k=m_1+1}^m \int_{ \Om_{k,n} } {2\al_k^2|y|^{\al_k-2}\over (1+|y|^{\al_k})^2} \Psi_{2,k,n}(y)\sum_{l=1,l\ne k}^m\gamma_{jj}^{-1} \gamma_{lj}  G(\xi_k+\de_k y,\xi_l)  +O(\de_j^{\al_j}|\gamma_{jj}|^{-1})
= o(1)
\end{split}
\end{equation*}
in view of
\begin{equation*}
\begin{split}
\sum_{k=m_1+1}^m &\int_{\Om_{k,n} } {2\al_k^2|y|^{\al_k-2}\over (1+|y|^{\al_k})^2} \Psi_{2,k,n}(y) \gamma_{jj}^{-1} \gamma_{kj}\lf({1\over 2\pi}\log|\de_k y|-H(\xi_k+\de_k y,\xi_k)\rg) \, dy\\
&=\sum_{k=m_1+1}^m \int_{\Om_{k,n} } {2\al_k^2|y|^{\al_k-2}\over (1+|y|^{\al_k})^2} \Psi_{2,k,n}(y) \lf({\gamma_{jj}^{-1} \gamma_{kj}\over 2\pi}\log\de_k +{\gamma_{jj}^{-1} \gamma_{kj}\over 2\pi}\log|y|-H(\xi_k,\xi_k)+O(\de_k |y|)\rg)\\
&=\sum_{k=m_1+1}^m\bigg( {\gamma_{jj}^{-1} \gamma_{kj}\over 2\pi}\log\de_k -H(\xi_k,\xi_k)\bigg)\int_{\Om_{k,n} } {2\al_k^2|y|^{\al_k-2}\over (1+|y|^{\al_k})^2} \Psi_{2,k,n}(y)\\
&\ \ \ + \sum_{k=m_1+1}^m  {\gamma_{jj}^{-1} \gamma_{kj}\over 2\pi} \int_{\Om_{k,n} } {2\al_k^2|y|^{\al_k-2}\over (1+|y|^{\al_k})^2} \Psi_{2,k,n}(y)\log|y|\, dy \\
&\ \ \ + \sum_{k=m_1+1}^m O\lf(\de_k\int_{\Om_{k,n} } {2\al_k^2|y|^{\al_k-1}\over (1+|y|^{\al_k})^2} \Psi_{2,k,n}(y)\, dy\rg)\\
&=o(1)+O(|\log\e|^{-1})+\sum_{k=m_1+1}^m O(\de_k)=o(1),
\end{split}
\end{equation*}
and similarly it follows that
$$\gamma_{jj}^{-1}  \int_{\Om_\e}  K_1\psi_2 P_\e Z_{0j}=o(1).$$
Therefore, from \eqref{tpz0j} and the previous computations we conclude that either for $i=1$ and $j=1,\dots,m_1$ or for $i=2$ and $j=m_1+1,\dots, m $ there holds
\begin{align}\label{rtpz0j}
o(1) =&\, \(\gamma_{jj}^{-1} + {1\over 2\pi} \log\de_j\) \int_{\Om_{j,n} } {2\al_j^2|y|^{\al_j-2}\over (1+|y|^{\al_j})^2 } \Psi_{i,j,n}(y)\, dy + {1\over 2\pi } \int_{\Om_{j,n} } {2\al_j^2|y|^{\al_j-2}\over (1+|y|^{\al_j})^2 } \log| y| \Psi_{i,j,n}(y)\, dy\\ 
&\,+o(1)\nonumber.
\end{align}
Notice that from \eqref{choice1} and \eqref{gamaij} we have that for any $j=1,\dots,m$
$$\gamma_{jj}^{-1} + {1\over 2\pi} \log\de_j= -{\log\e_j\over 4\pi} + O\Big({1\over |\log\e|^2}\Big) + {1\over 2\pi \al_j}\lf[\log\e +\log d_j\rg]=-{1\over 2\pi\al_j(\al_j-2) }\log\e +O(1).$$
Since we do not know the rate of the convergence $\ds \int_{\Om_{j,n} } {2\al_j^2|y|^{\al_j-2}\over (1+|y|^{\al_j})^2 } \Psi_{i,j,n}(y)\, dy=o(1)$
 for any $j=1,\dots, m$, we shall use the following rate
 \begin{equation}\label{0605}
 \sum_{j=1}^{m_1} \int_{\Om_{j,n} } {2\al_j^2|y|^{\al_j-2}\over (1+|y|^{\al_j})^2 } \Psi_{1,j,n}(y)\, dy
=  O(\e^{\sigma} )\quad\text{and}\quad
\sum_{j=m_1+1}^{m} \int_{\Om_{j,n} } {2\al_j^2|y|^{\al_j-2}\over (1+|y|^{\al_j})^2 } \Psi_{2,j,n}(y)\, dy=O(\e^{\sigma}).
\end{equation}
It is readily checked that
\begin{equation*}
\begin{split}
\int_{\Om_\e} K_1&=\sum_{k=1}^{m_1}\int_{\Om_\e} |x-\xi_k|^{\al_k-2}e^{U_k}\ dx 
=\sum_{k=1}^{m_1}\lf[4\pi\al_k + O(\de_k^{\al_k})+\sum_{i=1}^mO\Big({\e_i^2\over\de_i^2}\Big)\rg]
=\la_1+O(\e^{\ti\sigma})
\end{split}
\end{equation*}
and similarly
$$\int_{\Om_\e} K_2=\la_2 \tau^2+O(\e^{\ti\sigma})$$
for some $\ti\sigma>0$, so that for $\psi_1$ and $\psi_2$ we have that
$$\int_{\Om_\e} K_i\psi_i=\lf(1-{1\over \la_i\tau^{2(i-1)} }\int_{\Om_\e} K_i\rg)\int_{\Om_\e} K_i\phi=O(\e^{\ti\sigma})\int_{\Om_\e} K_i\phi=O(\e^{\ti\sigma}).$$
Also, we get that 
$$\int_{\Om_\e} K_1\psi_1 =\sum_{j=1}^{m_1} \int_{\Om_\e} |x-\xi_j|^{\al_j-2} e^{U_j} \psi_i = \sum_{j=1}^{m_1}\int_{ \Om_{j,n} } {2\al_j^2|y|^{\al_j-2}\over (1+|y|^{\al_j})^2 }  \Psi_{1,j,n}(y)\, dy$$
and 
$$\int_{\Om_\e} K_2\psi_2  = \sum_{j=m_1+1}^{m}\int_{ \Om_{j,n} } {2\al_j^2|y|^{\al_j-2}\over (1+|y|^{\al_j})^2 } \Psi_{2,j,n}(y)\, dy.$$
Hence, we deduce \eqref{0605}. 
Thus, multiplying \eqref{rtpz0j} by $-2\pi\al_j(\al_j - 2)$ and taking the sum either over $j=1,\dots,m_1$ for $i=1$ or $j=m_1+1,\dots,m$ for $i=2$ we conclude that
\begin{align*}
o(1) =&\, \sum_{j=1}^{m_1} \lf[ \(\log\e +O(1) \) \int_{\Om_{j,n} } {2\al_j^2|y|^{\al_j-2}\over (1+|y|^{\al_j})^2 } \Psi_{i,j,n}(y)\, dy -\al_j(\al_j - 2) \int_{\Om_{j,n} } {2\al_j^2|y|^{\al_j-2}\over (1+|y|^{\al_j})^2 } \log| y| \Psi_{i,j,n}(y)\, dy\rg]\\ 
= &\,   \log\e \sum_{j=1}^{m_1}  \int_{\Om_{j,n} } {2\al_j^2|y|^{\al_j-2}\over (1+|y|^{\al_j})^2 } \Psi_{i,j,n}(y)\, dy  + \sum_{j=1}^{m_1}O(1)  \int_{\Om_{j,n} } {2\al_j^2|y|^{\al_j-2}\over (1+|y|^{\al_j})^2 } \Psi_{i,j,n}(y)\, dy  \\
&\, - \sum_{j=1}^{m_1}\al_j(\al_j - 2)\int_{\Om_{j,n} } {2\al_j^2|y|^{\al_j-2}\over (1+|y|^{\al_j})^2 } \log| y| \Psi_{i,j,n}(y)\, dy
\end{align*}
and similarly
\begin{align*} 
o(1) =&\,   \log\e \sum_{j=m_1+1}^{m}  \int_{\Om_{j,n} } {2\al_j^2|y|^{\al_j-2}\over (1+|y|^{\al_j})^2 } \Psi_{i,j,n}(y)\, dy  + \sum_{j=m_1+1}^{m}O(1) \int_{\Om_{j,n} } {2\al_j^2|y|^{\al_j-2}\over (1+|y|^{\al_j})^2 } \Psi_{i,j,n}(y)\, dy \\
&\, - \sum_{j=m_1+1}^{m}\al_j(\al_j - 2)\int_{\Om_{j,n} } {2\al_j^2|y|^{\al_j-2}\over (1+|y|^{\al_j})^2 } \log| y| \Psi_{i,j,n}(y)\, dy.
\end{align*}
Therefore, passing to the limit we conclude that
$$\sum_{j=1}^{m_1}\al_j(\al_j - 2) a_j\int_{\R^2} {2\al_j^2|y|^{\al_j-2}\over (1+|y|^{\al_j})^2 } \log| y| Y_{0j}(y)\, dy=0$$
and
$$\sum_{j=m_1+1}^{m}\al_j(\al_j - 2) a_j\int_{\R^2} {2\al_j^2|y|^{\al_j-2}\over (1+|y|^{\al_j})^2 } \log| y| Y_{0j}(y)\, dy=0.$$
The first part of the claim follows since $\ds \int_{\R^2} {2\al_j^2|y|^{\al_j-2}\over (1+|y|^{\al_j})^2 } \log| y| Y_{0j}(y)\, dy=-4\pi.$

On the other hand, from claim \ref{claim3} we have that $a_j=-\dfrac{2}{\al_j-1 }\ti c_i$ either for $i=1$ and all $j=1,\dots, m_1$ or for $i=2$ and all $j=m_1+1,\dots,m$. Therefore, by replacing in \eqref{saj} we deduce that
$$0= 
-2\ti c_1 \sum_{j=1}^{m_1} \dfrac{\al_j(\al_j - 2)}{\al_j-1 } \qquad\text{and}\qquad
0= -2\ti c_2 \sum_{j=m_1+1}^{m} \dfrac{\al_j(\al_j - 2)}{\al_j-1 } .$$
Therefore, $\ti c_1=\ti c_2=0$ and consequently $a_j=0$ for all $j=1,\dots,m $, since $ \dfrac{\al_j(\al_j - 2)}{\al_j-1 } >0$.
\end{proof}

Now, using \eqref{cpsi} and Claim \ref{claim4}, we deduce that $\Phi_{j,n}\to 0$ weakly in $H_{\al_j}(\R^2)$ and strongly in $L_{\al_j}(\R^2)$ as $n\to+\infty$. Thus, we reach a contradiction with \eqref{snpsi}, and then the a-priori estimate $\|\phi\|\le C|\log \e | \, \|h\|_p$ is established. Concerning solvability issues, consider the space $H=H_0^1(\Om_\e)$ endowed with
the usual inner product $\ds [\phi,\psi]=\int_{\Om_\e} \grad\phi\grad\psi.$
Problem \eqref{pl} can be solved by finding $\phi\in H$ such that
$$[\phi,\psi]=\int_{\Om_\e} \lf[ K_1\lf(\phi -{1\over \la_1}\int_{\Om_\e} K_1\phi\rg) + K_1\lf(\phi -{1\over \la_2\tau^2}\int_{\Om_\e} K_2\phi\rg) -h\rg]\psi,\qquad\text{for all }\psi\in H.$$ With the aid
of Riesz's representation theorem, this equation gets rewritten in
$H$ in the operatorial form $\phi=\ml{K}(\phi)+\ti h$, for some
$\ti h\in H$, where $\ml{K}$ is a compact operator in $H$.
Fredholm's alternative guarantees unique solvability of this
problem for any $h$ provided that the homogeneous equation
$\phi=\ml{K}(\phi)$ has only the trivial solution in $H$. Since this is equivalent to \eqref{pl} with $h\equiv 0$, the existence of a unique solution follows from the a-priori estimate
\eqref{estphi}. The proof is complete.
\end{proof}

\bigskip

\begin{center}
\textsc{Acknowledgements}
\end{center}
\noindent Part of this work was carried out while the second
author was visiting the SBAI Department, University of Roma ``La Sapienza'' and the Department of Mathematics and Physics, University of ``Roma Tre''. He would like to express his gratitude to Prof. Pistoia and Prof. Esposito for the stimulating discussions and the warm hospitality. The first and the third author have been supported by MIUR Bando PRIN 2015 2015KB9WPT and  by GNAMPA as part of INdAM. The second author has been supported by grant Fondecyt Iniciaci\'on 11130517, Chile.

\end{document}